\documentclass[a4paper,11pt]{article}
\usepackage{authblk}
\usepackage[utf8]{inputenc}
\usepackage{graphicx}
\usepackage{subcaption}
\usepackage[margin=1in]{geometry}
\usepackage{bbm}
\usepackage{amssymb,amsfonts,amsmath,amsfonts,amsthm}
\usepackage{natbib}
\graphicspath{{./art/}}
\newtheorem{definition}{Definition}[section]
\newtheorem{theorem}{Theorem}[section]
\newtheorem{corollary}{Corollary}[theorem]
\newtheorem{lemma}{Lemma}[section]
\newtheorem{proposition}{Proposition}[section]
\theoremstyle{remark}
\newtheorem{remark}{Remark}[section]
\def\T{{ \mathrm{\scriptscriptstyle T} }}
\newcommand{\R}{\mathbb R}
\newcommand{\Sp}{\mathbb S}
\newcommand{\bmu}{\boldsymbol{\mu}}

\newcommand{\E}{\mathbf {E}}

\newcommand{\pr}{\mathbb {P}}
\newcommand{\N}{\mathbb {N}}

\newcommand{\sign}{\mathrm{sgn}}

\providecommand{\keywords}[1]
{
  \small	
  \textbf{\textit{Keywords---}} #1
}

\title{The entropy based goodness of fit tests for generalized von Mises-Fisher distributions and beyond}

\author[1]{Nikolai~Leonenko}
\author[2]{Vitalii~Makogin \thanks{Corresponding author (e-mail: vitalii.makogin@uni-ulm.de).  The research was partially supported by DFG Grant 390879134.}}
\author[1]{Mehmet~Siddik~Cadirci}
\affil[1]{School of Mathematics, Cardiff University, Senghennydd~Road,~Cardiff,~Wales,~UK,~CF24~4AG.}
\affil[2]{Institute of  Stochastics, Ulm University, Ulm, 08069 Germany.}

\date{\today}

\begin{document}
\maketitle

\begin{abstract}
We introduce some new classes of unimodal rotational invariant directional distributions, which generalize von Mises--Fisher distribution. We propose three types of distributions, one of which represents axial data. For each new type we provide formulae and short computational study of parameter estimators by the method of moments and the method of maximum likelihood.  The main goal of the paper is to develop the goodness of fit test to detect that sample entries follow one of the introduced generalized von Mises--Fisher distribution based on the maximum entropy principle. We use $k$th nearest neighbour distances estimator of Shannon entropy and prove its $L^2$-consistency. We examine the behaviour of the test statistics, find critical values and compute power of the test  on simulated samples. We apply the goodness of fit test to  local fiber directions in a glass fibre reinforced composite material and detect the samples which follow axial generalized von Mises--Fisher distribution.
\end{abstract}

\section{Introduction}
\label{sec:intro}
Directional distributions characterize randomness in unit vectors (directions). Spherical data sets appear in a wide range of problems arising from Earth sciences \citep{Pischiutta2013}, oceanography \citep{Wyatt97,shabani2016observations}, biology \citep{MOURITSEN2000}, physics \citep{Torrance66}, 
material science \citep{dresvyanskiy2020detecting}.

In this paper, we consider some classes of random unit vectors with values on sphere $\mathbb{S}^{d-1}=\{\mathbf{x}\in \R^d: \|\mathbf{x}\|=1\},$ which have the absolutely
continuous directional distributions with respect to the uniform distribution on $\Sp^{d-1}.$  We denote by $\mathbf{a}^{\T} \mathbf{b}$ the scalar product of vectors $\mathbf{a},\mathbf{b}\in \R^d$
and by $\|\mathbf{a}\|$ the Euclidean norm of $\mathbf{a}\in \R^d.$

The von Mises--Fisher distribution is a fundamental  isotropic distribution which is widely used in directional statistics \cite[e.g.][p. 168]{Mardia2000}. It belongs to the exponential family of distributions, is rotational invariant and has a density proportional to $\exp(\kappa \bmu^\T \mathbf{x}),\mathbf{x}\in \Sp^{d-1},$ such that random vectors are concentrated with rate $\kappa\in \R$ along direction $\bmu\in \Sp^{d-1}.$ Among several important properties of the von Mises--Fisher distribution we focus on maximum entropy characterization, that is, the von Mises--Fisher distribution has maximum entropy in the class of continuous distributions on $\Sp^{d-1}$ with a given value of $\E ({X})$  \citep[see][]{Mardia75}.

There are several generalizations, including the Fisher--Bingham distribution with a density proportional to 
$\exp(\kappa \bmu^\T \mathbf{x}+\mathbf{x}^\T A\mathbf{x}),\mathbf{x}\in \Sp^{d-1}$ \citep[see][]{Mardia75}, and  the generalized von Mises--Fisher
distribution of order $k$ (GvMFk) introduced in \citet{Gatto2007}, having
the density proportional to $\exp\left(\sum_{j=1}^k \kappa_j (\bmu_j \mathbf{x})^{r_j}\right),$ where $\bmu_j\in \Sp^{d-1},\kappa_j\in \R, r_j\in \N $  $(j=1,\ldots,k),$ and $r_1\leq\ldots\leq r_k.$ 

In this paper, we introduce a new generalization of the von Mises--Fisher
distribution, which stays in the exponential family and is rotational invariant with one mode. In contrast to the generalized von Mises--Fisher
distribution of order $k$ with integer power $r\in \N,$ we consider densities with arbitrary positive power $r\in \R_+.$ The motivation of such choice is to provide the analogue of a generalized Gaussian distribution for random vectors on the unit sphere. To do so we introduce the following three types of distributions of order $\alpha\in \R_+$, whose densities $f$ are proportional to 
\begin{align*}
     \text{ Type I, }& \mathrm{GvMF}_{1,d}(\alpha,\kappa,\bmu):&  f(\mathbf{x})&\propto\exp\left(\frac{\kappa}{\alpha} (\bmu^\T\mathbf{x})^{<\alpha>}\right),\mathbf{x}\in \Sp^{d-1},\\
    \text{ Type II, }& \mathrm{GvMF}_{2,d}(\alpha,\kappa,\bmu):&  f(\mathbf{x})&\propto\exp\left(\frac{\kappa}{2^{\alpha}\alpha} \|\mathbf{x}-\bmu\|^{2\alpha}\right), \mathbf{x}\in \Sp^{d-1},\\
    \text{ Axial Type, }& \mathrm{GvMF}_{3,d}(\alpha,\kappa,\bmu):&  f(\mathbf{x})&\propto\exp\left(\frac{\kappa}{\alpha} |\bmu^\T\mathbf{x}|^{\alpha}\right),\mathbf{x}\in \Sp^{d-1},
\end{align*}
where $\kappa>0$ is a concentration parameter, and $\bmu\in \Sp^{d-1}$ is a mean direction parameter.
In the paper we denote by $x^{<\alpha>}=|x|^{\alpha}\sign(x),x\in \R.$

Apart from studying the properties, simulations and parameter estimation for distributions $\mathrm{GvMF}_{j,d}$ $(j=1,2,3),$ we develop the goodness of fit test based on the estimation of the Shannon entropy and independent identically distributed (i.i.d.) sample. These tests exploit the maximum entropy principle, which is also proved in the paper as the spherical analogue of the results by \citet{Lutwak2004}.

To do so,  we employ the entropy estimators $\hat{H}_{N,k}$ derived from $k$th nearest neighbour distances.  Starting from the pioneering paper \citep{kozachenko1987}, which proves by direct probability methods the consistency of $\hat{H}_{N,1}$ for random vectors with values in Euclidean space, a large number of authors considered extending the class of admissible distributions and improved the convergence of $\hat{H}_{N,k},$ see \citet{berrett2019,Bulinski2019,delattreFournier2017,Evans2008,Evans2002,GaoOh2018,Goria2005,leonenkoPronzatoSavani2008}, and the references therein. In  \citep{li2011,MisraSH2010}, the $k$th nearest neighbour entropy estimation is generalized for hyperspherical distributions.

Unlike the above mentioned works, the limit theory for point processes with a fixed $k$ allows to prove the $L^p$-consistency of functionals of $k$th nearest neighbour distances for a wider class of distributions. The nearest neighbours method of estimation of the Shannon entropy for manifolds, including spheres, was developed by  \citet{PY3_2013}.
In the present paper, we continue their work and prove the $L^2$-consistency of $\hat{H}_{N,k},$ as $N\to \infty$ with arbitrary $k\geq 1$ and for a  random vector on a Riemannian manifold if it density is bounded and has compact support, see Theorem \ref{thm:L2}. Therefore, we show that $\hat{H}_{N,k}$ is a consistent estimator for the samples from the introduced generalized von Mises--Fisher distributions.

From the recent papers, we mention \citep{berrett2019}, where the efficient entropy estimation is provided via the weighted $k$th nearest neighbour distances with $k=k_N$ depending on sample size $N.$ Moreover, \citet{Berrett:bio19} introduced a non-parametric entropy based test of independence for multidimensional data.   \citet{Jammalamadaka2000} considered the entropy based test of goodness of fit for the von Mises distribution on the circle and use a different entropy estimate.  Our study is motivated, particularly, by the work of \citet{Mehmet}, where the entropy based goodness of fit test for generalized Gaussian distribution is given.

We verify our theoretical results by computational study on simulated samples and show the inflation of variances of $\hat{H}_{N,k}$ as $k$ grows, which confirms the conclusion of \citet{berrett2019}. Moreover, we detect the evidence of generalized von Mises--Fisher distributions in real world data by the presented entropy based goodness of fit test. Particularly, we find the evidence in  3D images of a glass fibre reinforced composite material, where fiber directions follow a generelized von Mises--Fisher distribution of axial type.

The manuscript is organized as follows. In Section 2, we revise the basic facts for von Mises-Fisher distribution. In Section 3,  we introduce our three types of generalized von Mises-Fisher distribution and compute their moments (Section 3.1). Section 4 is devoted to the Shannon entropy of generalized von Mises-Fisher distributions and we show the maximum entropy principle for them in Section 4.1. Then we discuss the statistical estimation of an entropy and prove the $L^2$ convergence of $k$th nearest neighbour estimator for random variables on compact manifolds (Section 4.2). In section 5, we formulate the maximum likelihood estimators (Section 5.1) and estimators by the method of moments for distributions $\mathrm{GvMF}_{1,d},\mathrm{GvMF}_{2,d},$  and $\mathrm{GvMF}_{3,d}$ (Section 5.2). In Section 6 we develop goodness of fit tests based on the maximum entropy principle for $\mathrm{GvMF}_{2,d}$ (Section 6.1) and $\mathrm{GvMF}_{1,d}, \mathrm{GvMF}_{3,d}$ (Section 6.2) distributions. Results of numerical experiments on simulated samples are given in Section 7. We present the method of simulations in Section 7.1, parameter estimation in Section 7.2, entropy estimation in Section 3.7, and study of the test statistics in Section 7.4. In Section 8, we detect the generalized von Mises-Fisher distributions in a real data set.   Some auxiliary material is given in the Appendix.






\section{Preliminaries}
In this section we provide some known facts needed for the sequel.
Let $\sigma(d\mathbf{x})$ be spherical measure on the sphere $\mathbb{S}^{d-1}.$ It can be written in polar coordinates $\mathbf{x}=(1,\mathbf{u}),1>0,\mathbf{u}\in \mathbb{S}^{d-1}$ as $$\sigma(d\mathbf{x})=\frac{\Gamma\left(\frac{d}{2}\right)}{2\pi^{d/2}} d \mathbf{u}.$$
\begin{lemma}[{\cite[Lemma 2.5.1]{KaiTai}}]
\label{lmm1}
Let $g:\R\to \R_+$ be a non-negative Borel function, $c>0$ and $\mathbf{a}=(a_1,\ldots,a_m)\neq \mathbf{0}.$ Then
\begin{equation}
    \label{lmm1:eq1}
    \int_{\mathbf{x}^\T\mathbf{x}=c^2} g(\mathbf{a}^\T \mathbf{x})\sigma(d\mathbf{x})=\frac{2c \pi^{\frac{d-1}{2}}}{\Gamma\left(\frac{d-1}{2}\right)}\int_{-c}^{c}g(\|a\|y)(c^2-y^2)^{\frac{d-3}{2}}dy.
\end{equation}
\end{lemma}
In particular, if $\mathbf{a}\in \Sp^{d-1},$ then 
\begin{equation}
    \label{lmm1:eq2}\int_{\mathbf{x}^\T\mathbf{x}=1} g(\mathbf{a}^\T \mathbf{x})\sigma(d\mathbf{x})=\frac{2 \pi^{\frac{d-1}{2}}}{\Gamma\left(\frac{d-1}{2}\right)}\int_{-1}^{1}g(y)(1-y^2)^{\frac{d-3}{2}}dy.
\end{equation}

\begin{definition}
\label{def:mf1}
A unit random vector $\pmb{\xi}$ has the $(d-1)$-dimensional {\it von Mises-Fisher distribution} $\mathrm{vMF}_{1,d}(\bmu,\kappa)$ if its probability density function with respect to the uniform distribution is 
\begin{equation}
    \label{def:mf1:eq} f_{\pmb{\xi}}(\mathbf{x})=\left(\frac{\kappa}{2}\right)^2 \frac{1}{\Gamma\left(\frac{d}{2}\right)I_{\frac{d}{2}-1}(\kappa)}\exp(\kappa \bmu^\T\mathbf{x}), \mathbf{x}\in  \Sp^{d-1},
\end{equation}
where $\kappa \in \R,$ $\bmu\in \Sp^{d-1}$ and $I_{\nu}$ is the modified Bessel function of order $\nu\geq 0.$
\end{definition}
In the case $d=3$ von Mises-Fisher distribution $M_{1,3}(\bmu,\kappa)$ is called {\it Fisher distribution}. In this case formula \eqref{def:mf1:eq} simplifies to $f_{\pmb{\xi}}(\mathbf{x})= \frac{\kappa}{2 \sinh \kappa}\exp(\kappa \bmu^\T\mathbf{x}), \mathbf{x}\in \Sp^{2}.$

The density of  $M_{1,d}(\bmu,\kappa)$ can be written in the alternative form.
\begin{definition}
\label{def:mf2}
A unit random vector $\pmb{\xi}$ has the $(d-1)$-dimensional {\it von Mises-Fisher distribution} $\mathrm{vMF}_{2,d}(\bmu,\kappa)$ if its probability density function with respect to the uniform distribution is 
\begin{equation}
    \label{def:mf2:eq} f_{\pmb{\xi}}(\mathbf{x})=\left(\frac{\kappa}{2}\right)^2 \frac{e^{\kappa}}{\Gamma\left(\frac{d}{2}\right)I_{\frac{d}{2}-1}(\kappa)}\exp\left(-\frac{\kappa}{2}\|\mathbf{x}-\bmu\|^2\right), \mathbf{x}\in  \Sp^{d-1},
\end{equation}
where $\kappa \in \R,$ $\bmu\in \Sp^{d-1}.$
\end{definition}
Indeed, $\frac{\kappa}{2}\|\mathbf{x}-\bmu\|^2=\frac{\kappa}{2}\|\mathbf{x}\|^2+\frac{\kappa}{2}\|\bmu\|^2-\kappa\bmu^\T \mathbf{x}=\kappa-\kappa\bmu^\T \mathbf{x}$ for $\mathbf{x},\bmu \in \Sp^{d-1}.$
Particularly, the density of  distribution $M_{2,3}(\bmu,\kappa)$ equals $\frac{\kappa}{1-\exp(-2\kappa)}\exp\left(-\frac{\kappa}{2}\|\mathbf{x}-\bmu\|^2\right), \mathbf{x}\in \Sp^{2}.$

Let us recall the standard directional statistics. 
\begin{definition}
\label{def:mdir}
Let $\mathbf{X}$ be random vector with values in $\Sp^{d-1}$ and $\E \mathbf{X}\neq \mathbf{0}.$ A mean direction of $\mathbf{X}$ is a vector $\frac{\E \mathbf{X}}{\|\E \mathbf{X}\|}.$  A mean resultant length is $\|\E \mathbf{X}\|.$
\end{definition}
The mean resultant length is invariant and the mean direction is equivariant under rotation. Formally, let $U\in SO(d)$ be a rotation matrix, then 
$\|\E U\mathbf{X}\|=\|\E \mathbf{X}\|$ and $\frac{\E U\mathbf{X}}{\|\E U\mathbf{X}\|}=U\frac{\E \mathbf{X}}{\|\E \mathbf{X}\|}.$

Consider the class of distributions on $\Sp^{d-1}$ with   rotational symmetry, that is their distribution functions have  a form  $f(\mathbf{x})=g(\bmu^\T\mathbf{x}),$ $\mathbf{x},\bmu\in \Sp^{d-1},$ e.g. \citep{bingham1975maximum}. Such random vectors $\mathbf{X}$ posses a {\it tangent-normal decomposition }\begin{equation}
    \label{tang:dec}
    \mathbf{X}=(\bmu^\T\mathbf{X})\bmu + \sqrt{1-(\bmu^\T\mathbf{X})^2}\mathbf{Y},
\end{equation}
where $\bmu^\T\mathbf{X}$ and $\mathbf{Y}$ are independent, $\bmu\perp \mathbf{Y},$ and $\mathbf{Y}$ is uniformly distributed on $\Sp^{d-2}.$
It follows from \eqref{tang:dec}, that the mean resultant length is $\|\E \mathbf{X}\|=\E[ \bmu^\T\mathbf{X}]$ and the mean direction equals $\bmu.$

\section{Generalized von Mises-Fisher distributions}

In this section we introduce our generalizations of the von Mises-Ficher distribution. We call $\kappa\in \R$ a concentration parameter and $\bmu\in \Sp^{d-1}$ a mean direction parameter.
\begin{definition}
\label{def:gmf1}
A unit random vector $\pmb{\xi}$ has the $(d-1)$-dimensional {\it I-type generalized von Mises-Fisher distribution} $\mathrm{GvMF}_{1,d}(\alpha,\kappa,\bmu)$ of order $\alpha>0$ if its probability density function with respect to the uniform distribution is 
\begin{equation}
    \label{def:gmf1:eq} f_{\pmb{\xi}}(\mathbf{x})=c_{1,d}(\kappa,\alpha)\exp\left(\frac{\kappa}{\alpha} (\bmu^\T\mathbf{x})^{<\alpha>} \right), \mathbf{x}\in  \Sp^{d-1},
\end{equation}
where 
\begin{equation}
    c_{1,d}(\kappa,\alpha)=\left(\frac{2 \pi^{\frac{d-1}{2}}}{\Gamma\left(\frac{d-1}{2}\right)}\int_{0}^{1}\left(e^{\frac{\kappa}{\alpha}y^\alpha}+e^{-\frac{\kappa}{\alpha}y^\alpha}\right)(1-y^2)^{\frac{d-3}{2}}dy\right)^{-1}.
\end{equation}
\end{definition}

As an analogue of von Mises-Fisher distribution in the form $\mathrm{vMF}_{2,d}$, we introduce the following class.
\begin{definition}
\label{def:gmf2}
A unit random vector $\pmb{\xi}$ has the $(d-1)$-dimensional {\it II-type generalized von Mises-Fisher distribution} $\mathrm{GvMF}_{2,d}(\alpha,\kappa,\bmu)$ of order $\alpha>0$ if its probability density function with respect to the uniform distribution is 
\begin{equation}
    \label{def:gmf2:eq} f_{\pmb{\xi}}(\mathbf{x})=c_{2,d}(\kappa,\alpha)\exp\left(-\frac{\kappa}{2^{\alpha}\alpha}\|\mathbf{x}-\bmu\|^{2\alpha}\right), \mathbf{x}\in  \Sp^{d-1},
\end{equation}
where 
\begin{equation}
c_{2,d}(\kappa,\alpha)=\left(\frac{2 \pi^{\frac{d-1}{2}}}{\Gamma\left(\frac{d-1}{2}\right)}\int_{0}^{1}\left(e^{-\frac{\kappa}{\alpha}\left(1-y\right)^{\alpha}}+e^{-\frac{\kappa}{\alpha}\left(1+y\right)^{\alpha}}\right)(1-y^2)^{\frac{d-3}{2}}dy\right)^{-1}.
\end{equation}
\end{definition}
In the case of $\alpha=1,$ the introduced distributions $\mathrm{GvMF}_{1,d}$ and $\mathrm{GvMF}_{2,d}$ become the von Mises-Fisher distributions $\mathrm{vMF}_{1,d}$ and $\mathrm{vMF}_{2,d}$ respectively.

If we do not distinguish opposite directions we deal with axec. Commonly used technique in this case is to consider symmetric density functions $f$ such that $f(\mathbf{x})=f(-\mathbf{x}),\mathbf{x}\in \Sp^{d-1}.$ Since our motivation is to stay in the class of rotational invariant densities and to generalize the von Mises-Fisher distribution, we propose the following model for an axial data.
\begin{definition}
A unit random vector $\pmb{\xi}$ has the $(d-1)$-dimensional  axial generalized von Mises-Fisher distribution $\mathrm{GvMF}_{3,d}(\alpha,\kappa,\bmu)$ (or distribution of axial type) of order $\alpha>0$ if its probability density function with respect to the uniform distribution is 
\begin{equation}
    \label{def:gmf3:eq} f_{\pmb{\xi}}(\mathbf{x})=c_{3,d}(\kappa,\alpha)\exp\left(\frac{\kappa}{\alpha}|\bmu^\T\mathbf{x}|^{\alpha}\right), \mathbf{x}\in  \Sp^{d-1},
\end{equation}
where $\kappa\in \R,$ $\bmu\in \Sp^{d-1}$ and
\begin{equation}
c_{3,d}(\kappa,\alpha)= \left(\frac{4 \pi^{\frac{d-1}{2}}}{\Gamma\left(\frac{d-1}{2}\right)}\int_{0}^{1}e^{\frac{\kappa}{\alpha}y^\alpha}(1-y^2)^{\frac{d-3}{2}}dy\right)^{-1}.
\end{equation}
\end{definition}
\begin{remark}
Let us check that $\int_{\Sp^{d-1}}f_{\pmb{\xi}}(\mathbf{x})\sigma(d \mathbf{x})=1.$ Then, the constant $c_{1,d}(\kappa,\alpha)$ equals
\begin{align*}
    &\left(\int_{\Sp^{d-1}}\exp\left(\frac{\kappa}{\alpha} (\bmu^\T\mathbf{x})^{<\alpha>}\right)\sigma(d \mathbf{x})\right)^{-1} \\
    & \stackrel{\eqref{lmm1:eq2}}{=}
    \left(\frac{2 \pi^{\frac{d-1}{2}}}{\Gamma\left(\frac{d-1}{2}\right)}\int_{-1}^{1}\exp\left(\frac{\kappa}{\alpha} y^{<\alpha>}\right)(1-y^2)^{\frac{d-3}{2}}dy\right)^{-1}\\
    &=\left(\frac{2 \pi^{\frac{d-1}{2}}}{\Gamma\left(\frac{d-1}{2}\right)}\int_{0}^{1}\left(e^{\frac{\kappa}{\alpha}y^\alpha}+e^{-\frac{\kappa}{\alpha}y^\alpha}\right)(1-y^2)^{\frac{d-3}{2}}dy\right)^{-1}.
\end{align*}
Similarly, the constant $c_{2,d}(\kappa,\alpha)$ equals
\begin{align*}
    &\left(\int_{\Sp^{d-1}}\exp\left(-\frac{\kappa}{2^{\alpha}\alpha}\|\mathbf{x}-\bmu\|^{2\alpha}\right)\sigma(d \mathbf{x})\right)^{-1} =\left(\int_{\Sp^{d-1}}\exp\left(-\frac{\kappa}{\alpha}\left(1-\bmu^\T \mathbf{x}\right)^{\alpha}\right)\sigma(d \mathbf{x})\right)^{-1}\\
    & \stackrel{\eqref{lmm1:eq2}}{=}
    \left(\frac{2 \pi^{\frac{d-1}{2}}}{\Gamma\left(\frac{d-1}{2}\right)}\int_{-1}^{1}\exp\left(-\frac{\kappa}{\alpha}\left(1-y\right)^{\alpha}\right)(1-y^2)^{\frac{d-3}{2}}dy\right)^{-1}\\
    &=\left(\frac{2 \pi^{\frac{d-1}{2}}}{\Gamma\left(\frac{d-1}{2}\right)}\int_{0}^{1}\left(e^{-\frac{\kappa}{\alpha}\left(1-y\right)^{\alpha}}+e^{-\frac{\kappa}{\alpha}\left(1+y\right)^{\alpha}}\right)(1-y^2)^{\frac{d-3}{2}}dy\right)^{-1}.
\end{align*}
\end{remark}
In the case $\alpha=2,$ the generalized von Mises-Fisher distribution of axial type reduces to the Watson distribution. 
\subsection{Moments}
In this section we consider moments characteristics of $\mathrm{GvMF}_{j,d}(\alpha,\kappa,\bmu), j=1,2,3$ distributions.
Denote by
\begin{align}
 \label{def:A1}   A_{1,d}(\kappa,\alpha,\beta)&=\int_{0}^{1}   e^{\frac{\kappa}{\alpha}y^\alpha}y^\beta(1-y^2)^{\frac{d-3}{2}}dy, \\
 \label{def:A2}   A_{2,d}(\kappa,\alpha,\beta)&=\int_{0}^{2} e^{-\frac{\kappa}{\alpha}y^{\alpha}} (2-y)^{\frac{d-3}{2}}y^{\frac{d-3}{2}+\beta}dy.
\end{align}
\begin{proposition}
\label{prop1}
Let $\beta\geq 0$ and $\mathbf{X}\sim \mathrm{GvMF}_{1,d}(\alpha,\kappa,\bmu),$ then
\begin{equation}
\label{GM1:m1}
    \E\left( (\bmu^\T\mathbf{X})^{<\beta>}\right)=\frac{ A_{1,d}(\kappa,\alpha,\beta)- A_{1,d}(-\kappa,\alpha,\beta)}    { A_{1,d}(\kappa,\alpha,0)+ A_{1,d}(-\kappa,\alpha,0)}.
\end{equation}
\end{proposition}
\begin{proof}
Let $f$ be the density of the form \eqref{def:gmf1:eq}, then
\begin{align*}
    &\E\left( (\bmu^\T\mathbf{X})^{<\beta>} \right)=\int_{\Sp^{d-1}} (\bmu^\T\mathbf{x})^{<\beta>} f(\mathbf{x})\sigma(d\mathbf{x})\\
    &\stackrel{\eqref{lmm1:eq2}}{=}c_{1,d}(\kappa,\alpha)\frac{2 \pi^{\frac{d-1}{2}}}{\Gamma\left(\frac{d-1}{2}\right)}\int_{-1}^{1} y^{<\beta>} \exp\left(\frac{\kappa}{\alpha} y^{<\alpha>} \right)(1-y^2)^{\frac{d-3}{2}}dy\\
    &=c_{1,d}(\kappa,\alpha)\frac{2 \pi^{\frac{d-1}{2}}}{\Gamma\left(\frac{d-1}{2}\right)}\int_{0}^{1}  \left( e^{\frac{\kappa}{\alpha}y^\alpha}-e^{-\frac{\kappa}{\alpha}y^\alpha} \right)y^\beta(1-y^2)^{\frac{d-3}{2}}dy\\
    &=\frac{\int_{0}^{1}  \left( e^{\frac{\kappa}{\alpha}y^\alpha}-e^{-\frac{\kappa}{\alpha}y^\alpha} \right)y^\beta(1-y^2)^{\frac{d-3}{2}}dy}
    {\int_{0}^{1}  \left( e^{\frac{\kappa}{\alpha}y^\alpha}+e^{-\frac{\kappa}{\alpha}y^\alpha} \right)(1-y^2)^{\frac{d-3}{2}}dy}.
\end{align*}
\end{proof}
Particularly, the mean direction of $\mathbf{X}\sim \mathrm{GvMF}_{1,d}(\alpha,\kappa,\bmu)$ is $\bmu$ and its mean resultant length   equals \begin{equation}
\label{GM2:m1}     \|\E[\mathbf{X}]\|=\E[\bmu^\T \mathbf{X}]=\frac{ A_{1,d}(\kappa,\alpha,1)- A_{1,d}(-\kappa,\alpha,1)}    { A_{1,d}(\kappa,\alpha,0)+ A_{1,d}(-\kappa,\alpha,0)}.
    \end{equation}
\begin{proposition}
\label{prop2}
Let $\beta\geq 0$ and $\mathbf{X}\sim \mathrm{GvMF}_{2,d}(\alpha,\kappa,\bmu),$ then
\begin{equation}
\label{prop2:eq1}
    \E \|\mathbf{X}-\bmu\|^{2\beta} =2^\beta\frac{A_{2,d}(\kappa,\alpha,\beta)}{A_{2,d}(\kappa,\alpha,0)}.
\end{equation}
\end{proposition}
\begin{proof}
Let $f$ be the density of the form \eqref{def:gmf2:eq}, then
\begin{align*}
    & \E \|\mathbf{X}-\bmu\|^{2\beta}=\int_{\Sp^{d-1}} \|\mathbf{x}-\bmu\|^{2\beta}f(\mathbf{x})\sigma(d\mathbf{x})=\int_{\Sp^{d-1}} (2-2\bmu^\T\mathbf{x})^{\beta}f(\mathbf{x})\sigma(d\mathbf{x})\\
    &\stackrel{\eqref{lmm1:eq2}}{=}c_{2,d}(\kappa,\alpha)\frac{2 \pi^{\frac{d-1}{2}}}{\Gamma\left(\frac{d-1}{2}\right)}\int_{-1}^{1} (2-2y)^\beta \exp\left(-\frac{\kappa}{\alpha} (1-y)^\alpha \right)(1-y^2)^{\frac{d-3}{2}}dy\\
        &=2^\beta\frac{\int_{-1}^{1}(1-y)^\beta e^{-\frac{\kappa}{\alpha}\left(1-y\right)^{\alpha}} (1-y^2)^{\frac{d-3}{2}}dy}{\int_{-1}^{1} e^{-\frac{\kappa}{\alpha}\left(1-y\right)^{\alpha}} (1-y^2)^{\frac{d-3}{2}}dy}\\
        &=2^\beta\frac{\int_{0}^{2} e^{-\frac{\kappa}{\alpha}z^{\alpha}} (2-z)^{\frac{d-3}{2}}z^{\frac{d-3}{2}+\beta}dz}{\int_{0}^{2} e^{-\frac{\kappa}{\alpha}z^{\alpha}} (2-z)^{\frac{d-3}{2}}z^{\frac{d-3}{2}}dz}. 
\end{align*}

\end{proof}
Particularly, the mean direction of $\mathbf{X}\sim \mathrm{GvMF}_{2,d}(\alpha,\kappa,\bmu)$ is $\bmu$ and its mean resultant length  equals
    \begin{equation}
  \label{GM2:m2}   
     \|\E[\mathbf{X}]\|=\E[\bmu^\T \mathbf{X}]=1-\frac{ A_{2,d}(\kappa,\alpha,1)}{ A_{2,d}(\kappa,\alpha,0)}.
    \end{equation}
\begin{proposition}
\label{prop3}
Let $\beta\geq 0$ and $\mathbf{X}\sim \mathrm{GvMF}_{3,d}(\alpha,\kappa,\bmu),$ then
\begin{equation}
    \E\left( |\bmu^\T\mathbf{X}|^\beta \right)=\frac{ A_{1,d}(\kappa,\alpha,\beta)}{ A_{1,d}(\kappa,\alpha,0)}.
\end{equation}
\end{proposition}
\begin{proof} 
Let $f$ be the density of the form \eqref{def:gmf3:eq}, then
\begin{align*}
    &\E\left( |\bmu^\T\mathbf{X}|^\beta \right)=\int_{\Sp^{d-1}} |\bmu^\T\mathbf{x}|^\beta f(\mathbf{x})\sigma(d\mathbf{x})\\
    &\stackrel{\eqref{lmm1:eq2}}{=}c_{3,d}(\kappa,\alpha)\frac{2 \pi^{\frac{d-1}{2}}}{\Gamma\left(\frac{d-1}{2}\right)}\int_{-1}^{1} |y|^\beta  \exp\left(\frac{\kappa}{\alpha} |y|^\alpha \right)(1-y^2)^{\frac{d-3}{2}}dy\\
    &=\frac{\int_{0}^{1}   e^{\frac{\kappa}{\alpha}y^\alpha}y^\beta(1-y^2)^{\frac{d-3}{2}}dy}
    {\int_{0}^{1}   e^{\frac{\kappa}{\alpha}y^\alpha}(1-y^2)^{\frac{d-3}{2}}dy}.
\end{align*}
\end{proof}
Note that, the mean direction of $\mathbf{X}\sim \mathrm{GvMF}_{3,d}(\alpha,\kappa,\bmu)$ is not defined and its mean resultant length $\|\E[\mathbf{X}]\|=\E[\bmu^\T \mathbf{X}]$ equals 0.

\section{Entropy}
In this section we find the entropy of generalized von Mises--Fisher distributions, and show the maximum entropy principle for them.
Then we discuss the statistical estimation of an entropy and prove the $L^2$ convergence of the $k$th nearest neigbour estimator for random variables on compact manifolds.

\subsection{Maximum entropy principle for generalized von-Mises Fisher distributions}

Recall that the entropy of a continuous random vector $\mathbf{X}\in \Sp^{d-1}$ with a density $f$ is
\begin{equation}
\label{Entropy:def1}
    H(\mathbf{X})=-\int_{\Sp^{d-1}} (\log f(\mathbf{x})) f(\mathbf{x}) \sigma(d\mathbf{x}).
\end{equation}
For a density version $f$ we denote its support by $\mathrm{supp} f=\{\mathbf{x}\in \Sp^{d-1}: f(\mathbf{x})>0\}.$ Clearly,  the integral in \eqref{Entropy:def1} is taken over $\mathrm{supp} f.$
\begin{theorem}
\label{thm:entropy1}
Let $\mathbf{X}_j\sim \mathrm{GvMF}_{j,d}(\alpha,\kappa,\bmu),$ $j=1,2,3,$  then 
\begin{align}
\label{HX1}
    H(\mathbf{X}_1)&=-\log c_{1,d}(\kappa,\alpha)- \frac{\kappa}{\alpha}\E\left( (\bmu^\T\mathbf{X}_1)^{<\alpha>}\right),\\
\label{HX2}    H(\mathbf{X}_2)&=-\log c_{2,d}(\kappa,\alpha) +\frac{\kappa}{2^{\alpha}\alpha}\E \|\mathbf{X}_2-\bmu\|^{2\alpha}\\
\label{HX3}    H(\mathbf{X}_3)&=-\log c_{3,d}(\kappa,\alpha) -\frac{\kappa}{\alpha}\E |\bmu^\T\mathbf{X}_3|^{\alpha}.
\end{align}
\end{theorem}
\begin{proof}
Let $\mathbf{X}_1$ have density $f_1,$ then the entropy of $\mathbf{X}_1$ equals
\begin{align*}
    &-\int_{\Sp^{d-1}} (\log f_1(\mathbf{x})) f_1(\mathbf{x})\sigma(d\mathbf{x})=-\log c_{1,d}(\kappa,\alpha)\int_{\Sp^{d-1}}f_1(\mathbf{x})\sigma(d\mathbf{x}) \\
    &- \frac{\kappa}{\alpha}\int_{\Sp^{d-1}}  (\bmu^\T\mathbf{x})^{<\alpha>} f_1(\mathbf{x})\sigma(d\mathbf{x})=-\log c_{1,d}(\kappa,\alpha)- \frac{\kappa}{\alpha}\E\left( (\bmu^\T\mathbf{X}_1)^{<\alpha>}\right).
\end{align*}
For $\mathbf{X}_2$ with density $f_2$ we have
\begin{align*}
    &H(\mathbf{X}_2)=-\int_{\Sp^{d-1}} (\log f_2(\mathbf{x})) f_2(\mathbf{x})\sigma(d\mathbf{x})=-\log c_{2,d}(\kappa,\alpha)\int_{\Sp^{d-1}}f_2(\mathbf{x})\sigma(d\mathbf{x}) \\
    &+\frac{\kappa}{2^{\alpha}\alpha}\int_{\Sp^{d-1}} \|\mathbf{x}-\bmu\|^{2\alpha}f(\mathbf{x})\sigma(d\mathbf{x})=
  -\log c_{2,d}(\kappa,\alpha) +\frac{\kappa}{2^{\alpha}\alpha}\E \|\mathbf{X}_2-\bmu\|^{2\alpha}.
\end{align*}
The case of $\mathbf{X}_3$ is similar.
\end{proof}

\begin{theorem}
\label{thm1}
Let a unit  random vector $\mathbf{Z}\in \Sp^{d-1}$ have a generalized von Mises-Fisher distribution $\mathrm{GvMF}_{1,d}(\alpha,\kappa,\bmu)$. Then $\mathbf{Z}$ has the maximum entropy value over all continuous random variables $\mathbf{X}$ on $\Sp^{d-1}$ with 
\begin{equation}
\label{thm1:eq1}
    \E\left( (\bmu^\T\mathbf{X})^{<\alpha>}\right)=\E\left( (\bmu^\T\mathbf{Z})^{<\alpha>}\right).
\end{equation}
\end{theorem}
\begin{proof}
Let $\mathbf{X}$ be a random unit vector on $\Sp^{d-1},d\geq 2$ such that \eqref{thm1:eq1} holds true. 
Let $f$ and $f^*$ be the densities of $\mathbf{X}$ and $\mathbf{Z}$ respectively. By Jensen's inequality, 
\begin{align*}
    &\int_{\Sp^{d-1}}f(\mathbf{x})\log f^*(\mathbf{x}) \sigma(d\mathbf{x})-\int_{\Sp^{d-1}}f(\mathbf{x})\log f(\mathbf{x}) \sigma(d\mathbf{x})=\int_{\Sp^{d-1}}f(\mathbf{x})\log \frac{f^*(\mathbf{x})}{f(\mathbf{x})} \sigma(d\mathbf{x})\\
    &\leq \log\left(\int_{\Sp^{d-1}}f(\mathbf{x}) \frac{f^*(\mathbf{x})}{f(\mathbf{x})} \sigma(d\mathbf{x})\right)=0
\end{align*}
with equality if and only if $f=f^*$ almost everywhere with respect to the Lebesgue measure on $\Sp^{d-1}.$ So,
\begin{equation}H(\mathbf{X})=-\int_{\Sp^{d-1}}f(\mathbf{x})\log f(\mathbf{x}) \sigma(d\mathbf{x})\leq -\int_{\Sp^{d-1}}f(\mathbf{x})\log f^*(\mathbf{x}) \sigma(d\mathbf{x}).
\end{equation}
In this case $f^*(\mathbf{x}) = \log c_{1,d}(\kappa,\alpha) + \frac{\kappa}{\alpha} (\bmu^\T\mathbf{x})^{<\alpha>}, \mathbf{x}\in \Sp^{d-1}$ and hence
\begin{align*}H(\mathbf{X})&\leq  -\int_{\Sp^{d-1}}f(\mathbf{x})\log f^*(\mathbf{x})\sigma(d \mathbf{x})=
  -\log c_{1,d}(\kappa,\alpha) -\frac{\kappa}{\alpha}\int_{\Sp^{d-1}} (\bmu^\T\mathbf{x})^{<\alpha >}f(\mathbf{x})\sigma(d\mathbf{x})\\
  &=
  -\log c_{1,d}(\kappa,\alpha) -\frac{\kappa}{\alpha}\E\left[ (\bmu^\T\mathbf{X})^{<\alpha>}\right]=-\log c_{1,d}(\kappa,\alpha) -\frac{\kappa}{\alpha}\E\left[ (\bmu^\T\mathbf{Z})^{<\alpha>}\right]\\
  &=-\log c_{1,d}(\kappa,\alpha) -\frac{\kappa}{\alpha}\int_{\Sp^{d-1}} (\bmu^\T\mathbf{x})^{<\alpha >}f^*(\mathbf{x})\sigma(d\mathbf{x})\\
  &=-\int_{\Sp^{d-1}}f^*(\mathbf{x})\log f^*(\mathbf{x})\sigma(d\mathbf{x})=H(\mathbf{Z}).
\end{align*}
\end{proof}
The maximum entropy principle for generalized von Mises-Fisher distribution of type II has the following form.
\begin{theorem}
\label{thm2}
Let a unit random vector $\mathbf{Z}\in \Sp^{d-1}$ have a generalized von Mises-Fisher  distribution $\mathrm{GvMF}_{2,d}(\alpha,\kappa,\bmu)$. Then $\mathbf{Z}$ has the maximum entropy value over all continuous random variables $\mathbf{X}$ on $\Sp^{d-1}$ with 
\begin{equation}
\label{thm2:eq1}
    \E \|\mathbf{X}-\bmu\|^{2\alpha} =\E\|\mathbf{Z}-\bmu\|^{2\alpha}.
\end{equation}
\end{theorem}
\begin{proof}
Let $f$ and $f^*$ be the densities of $\mathbf{X}$ and $\mathbf{Z}$ respectively. The proof is similar to Theorem \ref{thm1}. Indeed, $f^*(\mathbf{x}) = \log c_{2,d}(\kappa,\alpha) - \frac{\kappa}{2^{\alpha}\alpha} \|\mathbf{x}-\bmu\|^{2\alpha}, \mathbf{x}\in \Sp^{d-1}$ and hence
 \begin{align*}&H(\mathbf{X})\leq  -\int_{\Sp^{d-1}}f(\mathbf{x})\log f^*(\mathbf{x})\sigma(d \mathbf{x})=
  -\log c_{2,d}(\kappa,\alpha) +\frac{\kappa}{2^{\alpha}\alpha}\int_{\Sp^{d-1}} \|\mathbf{x}-\bmu\|^{2\alpha}f(\mathbf{x})\sigma(d\mathbf{x})\\
  &=
  -\log c_{2,d}(\kappa,\alpha) +\frac{\kappa}{2^{\alpha}\alpha}\E \|\mathbf{X}-\bmu\|^{2\alpha}=-\log c_{2,d}(\kappa,\alpha) +\frac{\kappa}{2^{\alpha}\alpha}\E \|\mathbf{Z}-\bmu\|^{2\alpha}\\
  &=-\log c_{2,d}(\kappa,\alpha) +\frac{\kappa}{2^{\alpha}\alpha}\int_{\Sp^{d-1}} \|\mathbf{x}-\bmu\|^{2\alpha}f^*(\mathbf{x})\sigma(d\mathbf{x})=-\int_{\Sp^{d-1}}f^*(\mathbf{x})\log f^*(\mathbf{x})\sigma(d\mathbf{x})=H(\mathbf{Z}).
\end{align*}
\end{proof}

\begin{theorem}
\label{thm3}
Let a unit random vector $\mathbf{Z}\in \Sp^{d-1}$ have an axial generalized von Mises-Fisher  distribution $\mathrm{GvMF}_{3,d}(\alpha,\kappa,\bmu)$. Then $\mathbf{Z}$ has the maximum entropy value over all  continuous random variables $\mathbf{X}$ on $\Sp^{d-1}$ with 
\begin{equation}
\label{thm3:eq1}
    \E |\bmu^\T\mathbf{X}|^{\alpha} =\E|\bmu^\T\mathbf{Z}|^{\alpha}.
\end{equation}
\end{theorem}
\begin{proof}
 The proof is similar to Theorems \ref{thm1} and \ref{thm2}. In this case, $f^*(\mathbf{x}) = \log c_{3,d}(\kappa,\alpha) + \frac{\kappa}{\alpha} |\bmu^\T\mathbf{x}|^{\alpha}, \mathbf{x}\in \Sp^{d-1}$ and 
 \begin{align*}&H(\mathbf{X})\leq  
  -\log c_{3,d}(\kappa,\alpha) -\frac{\kappa}{\alpha}\int_{\Sp^{d-1}} |\bmu^\T\mathbf{x}|^{\alpha}f(\mathbf{x})\sigma(d\mathbf{x})\\
  &=
  -\log c_{3,d}(\kappa,\alpha) -\frac{\kappa}{\alpha}\E |\bmu^\T\mathbf{X}|^{\alpha}=-\log c_{3,d}(\kappa,\alpha) -\frac{\kappa}{\alpha}\E |\bmu^\T\mathbf{Z}|^{\alpha}=H(\mathbf{Z}).
\end{align*}
\end{proof}

\subsection{Entropy estimation}
In this section we give the method of an entropy estimation for unit random vectors.  Actually, we extend the phase-state to the arbitrary compact Riemannian manifold.

Let $m,d\in \N,$ $m\leq d$, and $\mathcal{M}$ be a $m$-dimensional $C^1$ manifold embedded in $\R^d$ with the atlas $((U_i,g_i), i\in I_0),$ i.e., for each $y\in \mathcal{M}$ there exists an open subset $U_i$ of $\R^m$ and a continuously differentiable injection $g_i:U_i\to\R^d,$ such that $ y \in g_i(U) \subset \mathcal{M},$ and 
$g_i$ is an open map from $U_i$ to $\mathcal{M},$ and the linear map $g_i'(u)$ has full rank
for all $u \in U_i$.  

For bounded measurable $h:\mathcal{M}\to \R,$ the
integral $\int_{\mathcal{M}}h(y)\nu(dy)$ is defined by
$$\int_{\mathcal{M}}h(y) \nu(dy) =\sum_{i\in I_0}\int_{U_i}h(g_i(x))\psi_i(g_i(x))\sqrt{\det(J_{g_i}(x))'(J_{g_i}(x))}dx,$$
where $\nu$ is a $\sigma$-finite measure on $\mathcal{M},$ $J_{g_i}$ is the Jacobian of $g_i$ and $\{\psi_i,i\in I_0\}$ is the partition of unity, see \citep[pp. 3--4]{PY3_2013} and \citep[Chapter 2]{Berger} for more detailed setting.

Let $f:\mathcal{M} \to \R_+$ be a probability density of independent random elements $X,X_i,i\in \N$ with values in $\mathcal{M},$ i.e., $\int_{\mathcal{M}}f(x)\nu(dx)=1.$  Denote by $\mathcal{X}_N=\{X_1,\ldots,X_N\},$ $N\geq k$ the samples of the first $N$ elements. The entropy of $X$ equals $H(X)=-\int_{\mathcal{M}}\log(f(x))f(x)\nu(dx).$ Let $F$ be a finite subset of $\{X_i,i\geq k\}$ and $\rho_k(x,F)$ be the Euclidean distance between $x$ and its $k$th nearest neighbour in $F\setminus \{x\}$. Let $\gamma\approx 0.5772$ be the Euler–Mascheroni constant.

\begin{definition}
The $k-$nearest neighbour estimator ($k-$NNE) of the entropy $H(X)$ is defined by 
\begin{equation}\label{eq_of_H_N}
\widehat{H}_{N,k}(\mathcal{X}_N) = \frac{1}{N}\sum_{i=1}^{N}\log \left(\rho_k^m(X_i,\mathcal{X}_N) V_m  (N-1){e}^{-\psi(k)}\right),
\end{equation}
where  $$\psi(k)= \sum_{j=1}^{k-1}\frac{1}{j}-\gamma, \quad V_m=\frac{\pi^{m/2}}{\Gamma(1+m/2)}.$$
\end{definition}
For example, if $k=1$, then 
the nearest neighbour  distances estimator (NNE) is 
\begin{equation}\label{eq64}
\widehat{H}_{N,1}(\mathcal{X}_N) = \frac{m}{N}\sum_{i=1}^N \log \rho_{N,1,i} + \log V_m + \gamma + \log (N-1).
\end{equation}
 We start the proof of $L^2$-consistency of $\widehat{H}_{N,k}$ by writing 
down the particular case of Theorem 3.1 from \citep{PY3_2013} for the functional \begin{equation}
\xi(x,\mathcal{X}):=\log \left(e^{-\psi(k)}V_m \rho_k^m(x,\mathcal{X})\right).
\end{equation}
\begin{theorem}
\label{thmE1}
Let $k\geq 1$ put $q=1$ or $q=2.$ Suppose there exists $p\geq q$ such that 
\begin{equation}
\label{integr1}
\sup_{N\geq k} \E \left|\xi\left(N^{\frac{1}{m}} X_1,N^{\frac{1}{m}} \mathcal{X}_N\right)\right|^p<\infty.
\end{equation} Then as $N\to \infty$ we have $L^q$ convergence
\begin{equation}
    \frac{1}{N}\sum_{x\in \mathcal{X}_N}\xi\left(N^{\frac{1}{m}} x,N^{\frac{1}{m}} \mathcal{X}_N\right) \to \int_{\mathcal{M}} \E[ \xi\left(\mathbf{0},\mathcal{P}_{f(x)}\right)]f(x)\nu(dx),
\end{equation}
where $\mathcal{P}_{\lambda}$ denotes a homogeneous Poisson point process of intensity $\lambda>0$ in $\R^m$ (embedded in $\R^d$).
\end{theorem}

For the bounded random variables $X_i,i\geq 1$ and $\rho_k(x,\mathcal{X}_N),$ we generalize Lemma 7.8 from \citep{PY3_2013}, which was proved for the case $k=1.$ 
\begin{lemma}
\label{lmm12}
Let $f$ is bounded  and has compact support on $\mathcal{M},$ then for any $\delta\in (0,m)$ $$\sup_{N\geq k} \E\left[\rho_k^{\delta}\left(N^{\frac{1}{m}} X_1,N^{\frac{1}{m}} \mathcal{X}_N\right)\right]<\infty.$$
\end{lemma}

\begin{proof} 
The proof is very similar to  \cite[Lemma 7.8]{PY3_2013}.   
Recall that $\mathcal{M}$ has the atlas $((U_i,g_i), i\in I_0),$ where $I_0=\{1,\ldots,i_0\},$ and there exist $\delta_i,x_i,i\in I_0$ such that $\mathcal{M}\in \cup_{i\in I_0}B_{\delta_i}(y_i).$ 

Denote $A_i=B_{\delta_i}\setminus \cup_{j<i}B_{\delta_j}(y_j).$ Since $\mathrm{supp}(f)$ is bounded then there exist $i_0\in \N$ and constant $C>0$ such that 
\begin{align}
\nonumber    &\E \left[N^{\frac{\delta}{m}}\rho^\delta_k(X_1,\mathcal{X}_N)\right]= N^{\frac{\delta}{m}-1}\E\left(\sum_{x\in \mathcal{X}_N}\rho^\delta_k(x,\mathcal{X}_N)\right)\\
\label{lmm12:eq}    &\leq N^{\frac{\delta}{m}-1}\left[\sum_{i=1}^{i_0}\sum_{x\in A_i\cap \mathcal{X}_N}\rho^\delta_k(x,A_i\cap \mathcal{X}_N)+C\right].
\end{align}
Now we prove that for all finite $\mathcal{Y}\subset A_i$
$$\sum_{x\in \mathcal{Y}}\rho^\delta_k(x,\mathcal{Y})\leq C_i \left[\mathrm{card}( \mathcal{Y})\right]^{1-\frac{\delta}{m}}.$$
where $C_i>0.$
Let $\mathcal{Y}\subset A_i$ and $y_j\in \mathcal{Y}$ be a the $j$th nearest neighbour of $x\in \mathcal{Y}.$ Taking $z_j\in \mathcal{Y}$ such that $g_i^{-1}(z_j)$ to be $j$-th nearest neighbour of $g_i^{-1}(x)$ in $g_i^{-1}(\mathcal{Y}),$ we have from \cite[Lemma 4.1]{PY3_2013} that 
\begin{align*}
    \rho_k(x,\mathcal{Y})&=\max\{\|y_1-x\|,\ldots,\|y_k-x\|\}\leq \max\{\|z_1-x\|,\ldots,\|z_k-x\|\}\\
    &\leq C_i \max\left\{\left\|g_i^{-1}(z_1)-g_i^{-1}(x)\right\|,\ldots,\left\|g_i^{-1}(z_k)-g_i^{-1}(x)\right\|\right\}=C_i \rho_k(g_i^{-1}(x),g_i^{-1}(\mathcal{Y})).
\end{align*}
Thus, from \cite[Lemma 3.3]{yukich2006} we have for any $\delta \in (0,m)$ \begin{align*}
    \sum_{x\in A_i\cap \mathcal{X}_N}\rho^\delta_k(x,A_i\cap \mathcal{X}_N) \leq C_i [\mathrm{diam}(g_i^{-1}(A_i\cap \mathcal{X}_N))]^\delta [\mathrm{card}(g_i^{-1}(A_i\cap \mathcal{X}_N))]^{1-\frac{\delta}{m}} \leq \tilde{C}_i N^{1-\frac{\delta}{m}}.
\end{align*} 
Hence, the right hand side of \eqref{lmm12:eq} is bounded above uniformly.
\end{proof}
Now we prove the $L^2$-convergence of the $k$th nearest neighbour estimator
\begin{equation}
  \widehat{H}_{N,k}(\mathcal{X}_N)=  \frac{1}{N}\sum_{x\in \mathcal{X}_N}\xi\left(N^{\frac{1}{m}} x,N^{\frac{1}{m}} \mathcal{X}_N\right),
\end{equation}
which is an extension from the case $k=1$ to $k\geq 1$ of Theorem 2.4. from  \citep{PY3_2013}.
\begin{theorem}
\label{thm:L2}
Suppose  $f$ is bounded and has compact support.  Then for every fixed $k\geq 1$
\begin{equation}
\E\left[\widehat{H}_{N,k}(\mathcal{X}_N)-H(X)\right]^2 \to 0  \quad \text{ as } N\to \infty.
\end{equation}
\end{theorem} 

\begin{proof} 
We apply Theorem \ref{thmE1}. 
Repeating the lines of the proof of \citep[Theorem 3]{Mehmet}, we compute $\E[ \xi\left(\mathbf{0},\mathcal{P}_\lambda\right)],$ where 
$$\xi\left(\mathbf{0},\mathcal{P}_{\lambda}\right)=\log V_m -\psi(k) + m \log \rho_k(\mathbf{0},\mathcal{P}_\lambda).$$ The random variable  $\rho_k(\mathbf{0},\mathcal{P}_\lambda)$ is the distance to the $k$th point of $\mathcal{P}_\lambda$ from $\mathbf{0}.$
Therefore, 
\begin{align*}&\pr(\rho_k(\mathbf{0},\mathcal{P}_\lambda)\leq t)=\pr(\mathcal{P}_\lambda\cap B_{t}(\mathbf{0}) \geq k)\\
&=1-\sum_{j=0}^{k-1}\frac{1}{j!}e^{-\lambda|B_{t}(\mathbf{0})|}(\lambda|B_{t}(\mathbf{0})|)^j=1-\sum_{j=0}^{k-1}\frac{1}{j!}e^{-\lambda t^m V_m}(\lambda t^m V_m)^j, t\geq 0.
\end{align*} Then
\begin{align*}
    m\E[\log \rho_k(\mathbf{0},\mathcal{P}_\lambda)]&=\int_{0}^\infty \log t^m \frac{(\lambda V_m)^k  (t^m)^{(k-1)}}{(k-1)!}e^{-\lambda V_m t^m} m t^{m-1}dt\\
    &=-\log (\lambda V_m) + \int_{0}^\infty \log y \frac{y^{k-1}}{(k-1)!}e^{-y}dy\\
    &=-\log \lambda  -\log V_m  + \psi(k).
\end{align*}
Thus, $$\int_{\mathcal{M}} \E[ \xi\left(\mathbf{0},\mathcal{P}_{f(x)}\right)]f(x)\nu(dx)=-\int_{\mathcal{M}} (\log f(x)) f(x)\nu(dx)=H(X).$$

Second, we check condition \eqref{integr1}. Note that 
for every $\delta\in (0,1)$ and $p>1$ there exists $C>0$ such that $|\log t|^p\leq Ct^{-\delta}\mathbbm{1}_{[0,1]}(t)+C t^{\delta}\mathbbm{1}_{[1,\infty)}(t), t>0.$ Then  \begin{align}
    \nonumber\E \left|\xi\left(N^{\frac{1}{m}} X_1,N^{\frac{1}{m}} \mathcal{X}_N\right)\right|^p &\leq 2^{p-1}\left|\log V_m - \psi(k)\right|^p+2^{p-1} \E \left|\log \rho_k^m\left(N^{\frac{1}{m}} X_1,N^{\frac{1}{m}} \mathcal{X}_N\right) \right|^p\\
    \nonumber&\leq 2^{p-1}\left|\log V_m - \psi(k)\right|^p\\
\label{thm2:eq3}    &+2^{p-1} C\E  \rho_k^{-\delta}\left(N^{\frac{1}{m}} X_1,N^{\frac{1}{m}} \mathcal{X}_N\right) \mathbbm{1}_{[0,1]}\left(\rho_k^{\delta}\left(N^{\frac{1}{m}} X_1,N^{\frac{1}{m}} \mathcal{X}_N\right)\right) \\
\label{thm2:eq4}    &+2^{p-1} C\E  \rho_k^{\delta}\left(N^{\frac{1}{m}} X_1,N^{\frac{1}{m}} \mathcal{X}_N\right) \mathbbm{1}_{[1,\infty)}\left(\rho_k^{\delta}\left(N^{\frac{1}{m}} X_1,N^{\frac{1}{m}} \mathcal{X}_N\right)\right).
\end{align}
Term \eqref{thm2:eq3} is finite because
\begin{align}
    \nonumber&\sup_{N\geq k} \E  \rho_k^{-\delta}\left(N^{\frac{1}{m}} X_1,N^{\frac{1}{m}} \mathcal{X}_N\right) \mathbbm{1}_{[0,1]}\left(\rho_k^{\delta}\left(N^{\frac{1}{m}} X_1,N^{\frac{1}{m}} \mathcal{X}_N\right)\right) \\
    \label{cond2}&\leq \sup_{N\geq k} \E  \rho_1^{-\delta}\left(N^{\frac{1}{m}} X_1,N^{\frac{1}{m}} \mathcal{X}_N\right) <\infty,
\end{align}
where \eqref{cond2} is ensured  by \cite[Lemma 7.5]{PY3_2013} if $f$ is bounded and $\delta\in (0,m).$ 
Hence, for \eqref{integr1} to be satisfied,  it ramains to show that 
\begin{equation}
    \label{cond1}\sup_{N\geq k} \E \rho_k^{\delta}\left(N^{\frac{1}{m}} X_1,N^{\frac{1}{m}} \mathcal{X}_N\right)<\infty.
\end{equation}

Thus, applying Lemma \ref{lmm12} we get that \eqref{cond1} holds true if $0<\delta <m.$ 
\end{proof}

The 2-dimensional sphere $\Sp^2$ is a compact manifold with $d=3,$ $m=2$ and $\nu=\sigma.$ Thus, Theorem \ref{thm2} is valid for all bounded densities on $\Sp^2,$ $k$th nearest neighbour estimator has the form
\begin{equation*}
\widehat{H}_{N,k}(\mathcal{X}_N) = \frac{2}{N}\sum_{i=1}^{N}\log \rho_k(X_i,\mathcal{X}_N)-\psi(k)+ \log(N-1) + \log \pi,
\end{equation*}
and $\widehat{H}_{N,k}(\mathcal{X}_N)\to H(X)$ in $L^2(\Omega).$ This yields, that $\widehat{H}_{N,k}(\mathcal{X}_N)$ is a consistent estimator of the Shannon entropy.

\section{Estimation of parameters}

\label{est1}

\subsection{Fisher's maximum likelihood estimation}
\label{MLE}
Let $\mathcal{X}_N=\{\mathbf{x}_1,\ldots,\mathbf{x}_N\}$ be a random sample. We write down the log-likelihood $l(\mathcal{X}_N)$ for random samples from the introduced generalized von Mises-Fisher distributions. 
\begin{lemma}
Let $\mathcal{X}_{j,N}\sim \mathrm{GvMF}_{j,d}(\alpha,\kappa,\bmu),$ $j=1,2,3$ then 
\begin{align}
\label{log1}l(\mathcal{X}_{1,N})&=N \log c_{1,d}(\kappa,\alpha)+\frac{\kappa}{\alpha}\sum_{i=1}^N(\bmu^\T \mathbf{x}_i)^{<\alpha>},\\
       \label{log2} l(\mathcal{X}_{2,N})&=N \log c_{2,d}(\kappa,\alpha)-\frac{\kappa}{2^\alpha \alpha}\sum_{i=1}^N\| \mathbf{x}_i-\bmu\|^{2\alpha},\\
      \label{log3}  l(\mathcal{X}_{3,N})&=N \log c_{3,d}(\kappa,\alpha)+\frac{\kappa}{\alpha}\sum_{i=1}^N|\bmu^\T \mathbf{x}_i|^\alpha.
\end{align}
\end{lemma}
\begin{proof}
The statements follow from the direct calculation of $l(\mathcal{X}_N).$
\end{proof}
In each case, we use numerical methods  to find the maximum likelihood estimates $(\hat{\mu}_{L},\hat{\kappa}_{L},\hat{\alpha}_{L})$ of $(\mu,\kappa,\alpha)$ which maximize the log-likelihoods \eqref{log1}-\eqref{log3}.

The problem becomes easier when parameter $\alpha$ is known. In such a case,
we can derive maximum likelihood estimates taking derivatives of \eqref{log1}-\eqref{log3}. In such a case, we have the estimates of $\mu$ as $\kappa$ as
\begin{itemize}
    \item Let $\mathcal{X}_N\sim \mathrm{GvMF}_{1,d}(\alpha,\kappa,\bmu),$ then 
    \begin{align*}
    &\hat{\bmu}_{L}=\arg\max_{\bmu\in \Sp^{d-1}}\sum_{i=1}^N(\bmu^\T \mathbf{x}_i)^{<\alpha>},\quad \frac{A_{1,d}(\hat{\kappa}_L,\alpha,\alpha)-A_{1,d}(-\hat{\kappa}_L,\alpha,\alpha)}{A_{1,d}(\hat{\kappa}_L,{\alpha},0)+A_{1,d}(-\hat{\kappa}_L,\alpha,0))}=\frac{1}{N}\sum_{i=1}^N(\bmu^\T_{L} \mathbf{x}_i)^{<\alpha>}.
    \end{align*}
    \item Let $\mathcal{X}_N\sim \mathrm{GvMF}_{2,d}(\alpha,\kappa,\bmu),$ then 
    \begin{align}
    &\hat{\bmu}_{L}=\arg \min_{\bmu\in \Sp^{d-1}}\sum_{i=1}^N\| \mathbf{x}_i-\bmu\|^{2\alpha},\quad 2^\alpha\frac{A_{2,d}(\hat{\kappa}_L,\alpha,\alpha)} {A_{2,d}(\hat{\kappa}_L,\alpha,0)}=\frac{1}{N}\sum_{i=1}^N\| \mathbf{x}_i-\hat{\bmu}_{L}\|^{2\alpha}.
    \end{align}
    \item Let $\mathcal{X}_N\sim \mathrm{GvMF}_{3,d}(\alpha,\kappa,\bmu),$ then 
    \begin{align}
    &\hat{\bmu}_{L}=\arg \max_{\bmu\in \Sp^{d-1}}\sum_{i=1}^N|\bmu^\T \mathbf{x}_i|^\alpha,\quad \frac{A_{1,d}(\hat{\kappa}_L,\alpha,\alpha)}{A_{1,d}(\hat{\kappa}_L,\alpha,0)}=\frac{1}{N}\sum_{i=1}^N|\hat{\bmu}_{L}^\T \mathbf{x}_i|^\alpha.
    \end{align}
\end{itemize}

\subsection{Method of moments}
\label{MME}
In this section we consider parameter estimation of generilized von Mises-Fisher distributions based on moments estimation.

In the case of non-axial random vector $\mathbf{X}\in \Sp^{d-1}$ we assume that $\|\E \mathbf{X}\|\neq 0.$ We know from Definition \ref{def:mdir} that $\hat{\bmu}:=\frac{\bar{\mathcal{X}}_N}{\|\bar{\mathcal{X}}_N\|},$ where $\bar{\mathcal{X}}_N=\frac{1}{N}\sum_{i=1}^N\mathbf{x}_i,$ is the natural estimator for mean direction parameter $\bmu$. In order to find estimates for parameters  $\alpha$ and $\kappa$ we need at least two more moment statistics. 
A
standard approach involves the use of the resultant length. For the second relations we choose $\E (\sign\{{\mathbf{X}_1}^{\T}\E {\mathbf{X}_1}\})$ and $\E(\left\| \mathbf{X}_2-\bmu\right\|^4)$ for vectors $\mathbf{X}_j\sim\mathrm{GvMF}_{j,d}(\alpha,\kappa,\mu)$ $(j=1,2).$ 
Applying  \eqref{GM2:m1} and \eqref{GM2:m2}, one can get  the estimators $\hat{\kappa},\hat{\alpha}$ as a solution of the following equations
\begin{align*}
 &\frac{A_{1,d}(\hat{\kappa},\hat{\alpha},1)- A_{1,d}(-\hat{\kappa},\hat{\alpha},1)}{ A_{1,d}(\hat{\kappa},\hat{\alpha},0)+ A_{1,d}(-\hat{\kappa},\hat{\alpha},0)}=\|\bar{\mathcal{X}}_{1,N}\|,
\frac{A_{1,d}(\hat{\kappa},\hat{\alpha},0)- A_{1,d}(-\hat{\kappa},\hat{\alpha},0)}{ A_{1,d}(\hat{\kappa},\hat{\alpha},0)+ A_{1,d}(-\hat{\kappa},\hat{\alpha},0)}=\frac{\sum_{i=1}^N \sign(x_i^{\T} \bar{\mathcal{X}}_{1,N})}{N\|\bar{\mathcal{X}}_{1,N}\|};\\
&\frac{A_{2,d}(\hat{\kappa},\hat{\alpha},1)}{ A_{2,d}(\hat{\kappa},\hat{\alpha},0)}=1-\|\bar{\mathcal{X}}_{2,N}\|,\quad 
\frac{A_{2,d}(\hat{\kappa},\hat{\alpha},2)}{ A_{2,d}(\hat{\kappa},\hat{\alpha},0)}=\frac{1}{4N}\sum_{i=1}^N\left\| {x}_i-\frac{\bar{\mathcal{X}}_{2,N}}{\|\bar{\mathcal{X}}_{2,N}\|}\right\|^4
\end{align*}
for  the samples $\mathcal{X}_{1,N}\sim\mathrm{GvMF}_{1,d}(\alpha,\kappa,\bmu)$ and $\mathcal{X}_{2,N}\sim\mathrm{GvMF}_{2,d}(\alpha,\kappa,\bmu),$ respectively. 
\begin{remark}
If the parameter $\alpha$ is known, we can reduce the problem of the moment estimation to the solution of one equation. Namely, 
$$\frac{A_{1,d}(\hat{\kappa},\alpha,1)- A_{1,d}(-\hat{\kappa},\alpha,1)}{ A_{1,d}(\hat{\kappa},\alpha,0)+ A_{1,d}(-\hat{\kappa},\alpha,0)}=\|\bar{\mathcal{X}}_{1,N}\|, \text{ and } \frac{A_{2}(\hat{\kappa},\alpha,1)}{ A_{2}(\hat{\kappa},\alpha,0)}=1-\|\bar{\mathcal{X}}_{2,N}\|.$$
\end{remark}



In the case of a symmetrically distributed random vector $\mathbf{X}\in \Sp^{d-1},$ $\E \mathbf{X}=0$ and the value ${\E \mathbf{X}}/{\|\E \mathbf{X}\|}$ is not defined.
 Recall  the tangent-normal decomposition \eqref{tang:dec} of a random vector $\mathbf{X}\in \Sp^{d-1},$ that is
$\mathbf{X}=\bmu \xi +\sqrt{1-\xi^2} \mathbf{Y},$ where $\bmu\in \Sp^{d-1}$ is a mean direction parameter, $\xi$ is a random variable on $[-1,1]$ independent of a uniformly distributed random vector $Y\in \Sp^{d-2}$ such that $\bmu \perp\mathbf{Y}.$  

To find relations which determine the parameter $\bmu$ and distribution $\xi$ we consider an orientation tensor $T(\mathbf{X})$ given by
\begin{equation}
\label{def:tensor}
    T(\mathbf{X})=\mathbf{X}\mathbf{X}^\T=\xi^2 \bmu\bmu^\T+\xi \sqrt{1-\xi^2} \left(\bmu \mathbf{Y}^\T+\mathbf{Y} \bmu^\T\right)+(1-\xi^2)\mathbf{Y}\mathbf{Y}^\T.
\end{equation}

Therefore, the mean orientation tensor is 
\begin{equation}
    \E T(\mathbf{X})=\E[\mathbf{X}\mathbf{X}^\T]= \bmu\bmu^\T\E\xi^2+(1-\E \xi^2)\E \mathbf{Y}\mathbf{Y}^\T.
\end{equation}

\begin{theorem}
\label{thm:moma}
Let a random vector $\mathbf{X}$ has a representation as above, i.e., $\mathbf{X}=\bmu \xi +\sqrt{1-\xi^2} \mathbf{Y}.$ Then
\begin{align}
    \label{thm:moma:1}
    \bmu \bmu^\T&=\sqrt{\frac{d-1}{d \E[\mathbf{X}(\E \mathbf{X}\mathbf{X}^{\T})\mathbf{X}^{\T}]-1}}\left(\E \mathbf{X}\mathbf{X}^{\T}-\frac{1}{d}I_d\right)+\frac{1}{d}I_d,\\
    \label{thm:moma:2}
    \E \xi^2&=\frac{1}{d}+\sqrt{\frac{d-1}{d}}\sqrt{\E[\mathbf{X}(\E \mathbf{X}\mathbf{X}^{\T})\mathbf{X}^{\T}] -\frac{1}{d}},\\
    \label{thm:moma:3}
    \E \xi^4&=\E\left[\mathbf{X}^{\T}\E T(\mathbf{X})\mathbf{X}-\frac{1-\E \xi^2}{d-1}\right]^2 \frac{d-1}{d \E[\mathbf{X}(\E \mathbf{X}\mathbf{X}^{\T})\mathbf{X}^{\T}]-1}.
\end{align}
where $I_d$ is $d\times d$ identity matrix.
\end{theorem}
\begin{proof}
Let $U_{\bmu}\in SO(d),$ such that $\bmu=U_{\bmu} \mathbf{e}_x,$ where $\mathbf{e}_x=(1,0,\ldots,0)^{\T}.$ Denote by $\tilde{\mathbf{Y}}=U_{\bmu}^{-1}\mathbf{Y}.$ The vector $\tilde{\mathbf{Y}}$ is uniformly distributed on $\Sp^{d-2}$ with the first coordinate equal 0. Then $\mathbf{X}=U_{\bmu}\left(\xi \mathbf{e}_x+\sqrt{1-\xi^2}\tilde{\mathbf{Y}}\right)$ and 
$\E \mathbf{X}\mathbf{X}^{\T}=U_{\bmu}\left(\mathbf{e}_x\mathbf{e}_x^{\T} \E \xi^2+(1-\E \xi^2)\E \tilde{\mathbf{Y}}\tilde{\mathbf{Y}}^{\T}\right)U_{\bmu}^{-1}.$
 It follows from the symmetry that $\E \tilde{\mathbf{Y}} \tilde{\mathbf{Y}}^{\T}=\frac{1}{d-1}\left(\begin{matrix}0&0\\0&I_{d-1}\end{matrix}\right).$ 
Therefore,
\begin{equation}
\label{thm:moma:eq1}
    \E \mathbf{X}\mathbf{X}^{\T} = U_{\bmu}\left(\mathbf{e}_x\mathbf{e}_x^{\T}\E \xi^2+\frac{1-\E \xi^2}{d-1}(I_d-\mathbf{e}_x\mathbf{e}^{\T}_x)\right)U_{\bmu}^{\T}=\E \xi^2 \bmu \bmu^{\T} + \frac{1-\E \xi^2}{d-1}(I_d-\bmu \bmu^{\T}).
\end{equation}
Thus, $\bmu \bmu^{\T}=\frac{d-1}{d \E\xi^2-1}\left(\E \mathbf{X}\mathbf{X}^{\T}-\frac{1-\E \xi^2}{d-1}\right).$ Then consider $\mathbf{X}^{\T}\E T(\mathbf{X})\mathbf{X}.$
From \eqref{thm:moma:eq1} we have 
$$\mathbf{X}^{\T}\E T(\mathbf{X})\mathbf{X}= \mathbf{X}^{\T}\bmu \bmu^{\T}\mathbf{X}\E\xi^2+\frac{1-\E \xi^2}{d-1}(1-\mathbf{X}^{\T}\bmu \bmu^{\T}\mathbf{X})=\xi^2\E\xi^2+\frac{1-\E \xi^2}{d-1}(1-\xi^2).$$
and  $\E[\mathbf{X}^{\T}\E T(\mathbf{X})\mathbf{X}]= (\E\xi^2)^2+\frac{(1-\E \xi^2)^2}{d-1}.$ This yields
$\E \xi^2=\frac{1}{d}+\sqrt{\frac{d-1}{d}}\sqrt{\E[\mathbf{X}(\E \mathbf{X}\mathbf{X}^{\T})\mathbf{X}^{\T}] -\frac{1}{d}}$ and 
$$\bmu \bmu^{\T}=\sqrt{\frac{d-1}{d \E[\mathbf{X}(\E \mathbf{X}\mathbf{X}^{\T})\mathbf{X}^{\T}]-1}}\left(\E \mathbf{X}\mathbf{X}^{\T}-\frac{1}{d}I_d\right)+\frac{1}{d}I_d.$$
Finally,
$$\E\left[\mathbf{X}^{\T}\E T(\mathbf{X})\mathbf{X}-\frac{1-\E \xi^2}{d-1}\right]^2= \left(\frac{d\E \xi^2-1}{d-1}\right)^2\E \xi^4.$$
\end{proof}

Note, that for an axial vector $\mathbf{X},$ the random variable $\xi$ has a symmetric distribution on $[-1,1],$ therefore $\E \xi=0.$ So, if $\xi$ has two-dimensional parametric distribution, one get from Theorem \ref{thm:moma} the parameter estimates for $\xi$.

We apply Theorem \ref{thm:moma} for the random sample $\mathcal{X}_N=\{\mathbf{x}_1,\ldots,\mathbf{x}_N\}$  from $\mathrm{GvMF}_{3,d}(\alpha,\kappa,\bmu).$  Denote by $ \bar{T}=\frac{1}{n}\sum_{i=1}^n {\mathbf{x}_i}{\mathbf{x}_i}'.$ 
Then  $\E \xi^2=\E (\bmu' \mathbf{x}_1)^2$ and $\E \xi^4=\E (\bmu' \mathbf{x}_1)^4$ are  given in Proposition \ref{prop3}.

\begin{corollary}
\label{cor:moma2}
Let $\mathcal{X}_N\sim \mathrm{GvMF}_{3,d}(\alpha,\kappa,\bmu)$ and denote by $ \bar{T}=\frac{1}{N}\sum_{i=1}^N {\mathbf{x}_i}{\mathbf{x}_i}^{\T}$ and $\bar{V}=\frac{1}{N}\sum_{i=1}^N \mathbf{x}_i\bar{T_1}\mathbf{x}_i^{\T}.$ 
Then
\eqref{thm:moma:1}--\eqref{thm:moma:3} hold true and 
\begin{equation}
\label{cor:moma2:eq} 
\begin{gathered}
    \hat{\bmu}\hat{\bmu}^{\T}=\sqrt{\frac{d-1}{d \bar{V}-1}}\left(\bar{T}-\frac{1}{d}I_d\right)+\frac{1}{d}I_d,\quad \frac{A_{1}(\hat{\kappa},\hat{\alpha},2)}{A_{1}(\kappa,\alpha,0)}=\frac{1}{d}+\sqrt{\frac{d-1}{d}}\sqrt{\bar{V} -\frac{1}{d}},\\
    \frac{A_{1}(\kappa,\alpha,4)}{A_{1}(\kappa,\alpha,0)}=\frac{d-1}{(d \bar{V}-1)N}\sum_{i=1}^N\left({\mathbf{x}_i}^{\T}\bar{T}{\mathbf{x}_i}-\frac{1}{d}-\frac{\sqrt{d \bar{V}-1}}{d\sqrt{d-1}}\right)^2 .
\end{gathered}
\end{equation}
\end{corollary}




\section{Goodness of fit test based on the maximum entropy principle}
\label{Sec:GFtest}
In this section we provide the statistical test for verification that a random sample follows a generalized von Mises-Fisher distribution. The methodology for all three introduced distributions is very similar. For simplicity, we provide detailed explanation for the Type-II distribution.

\subsection{Type II}
Denote by $\mathrm{GvMF}_{2,d}$ the class of generalized von Mises-Fisher distributions $\mathrm{GvMF}_{2,d}(\alpha,\kappa,\bmu),$ $\alpha>0,$ $\kappa> 0$ and $\bmu\in \Sp^{d-1}.$
Let  $\mathcal{X}_N=\{\mathbf{x}_1,\ldots,\mathbf{x}_N\}$  be a random sample of vectors on a sphere $\Sp^{d-1}$ and $\mathbf{x}_j\stackrel{d}{=}\mathbf{X},j=1,\ldots,N$ with unknown distribution. 

Let $\mathbf{Z}\sim \mathrm{GvMF}_{2,d}(\alpha,\kappa,\bmu).$ From \eqref{def:A2}, \eqref{prop2:eq1},  and Theorem \ref{thm2} we know that $H(\mathbf{Z})\geq H(\mathbf{X})$ for all continuous random vectors $\mathbf{X}\in \Sp^{d-1}$ with $\E\|\mathbf{X}-\bmu\|^{2\alpha}=\E\|\mathbf{Z}-\bmu\|^{2\alpha}=2^\alpha\frac{A_{2,d}(\kappa,\alpha,\alpha)}{A_{2,d}(\kappa,\alpha,0)}.$ Using Theorem \ref{thm:entropy1}, we get
\begin{equation}
\label{test:eq1}
    \inf_{\substack{\alpha,\kappa>0,\\\bmu\in \Sp^{d-1}}}\left\{\left.-\log c_{2,d}(\kappa,\alpha) +\frac{\kappa }{\alpha}\frac{A_{2,d}(\kappa,\alpha,\alpha)}{A_{2,d}(\kappa,\alpha,0)}\right| \,
     \E\|\mathbf{X}-\bmu\|^{2\alpha}=2^\alpha\frac{A_{2,d}(\kappa,\alpha,\alpha)}{A_{2,d}(\kappa,\alpha,0)}\right\}\geq H(\mathbf{X}).
\end{equation}

Moreover, equality in \eqref{test:eq1} appears if and only if $\mathbf{X}$ belongs to some distribution from the family $\mathrm{GvMF}_{2,d}.$ 
We substitute now the unobservable value of $\E\|\mathbf{X}-\bmu\|^{2\alpha}$ by its statistical counterpart $\frac{1}{N}\sum_{i=1}^N\|\mathbf{x}_i-\bmu\|^{2\alpha}$ and define the statistics
\begin{equation*}
S_2(\mathcal{X}_N)= \inf_{\substack{\alpha,\kappa>0,\\\bmu\in \Sp^{d-1}}}\left\{\left.-\log c_{2,d}(\kappa,\alpha) +\frac{\kappa }{\alpha}\frac{A_{2,d}(\kappa,\alpha,\alpha)}{A_{2,d}(\kappa,\alpha,0)}\right| \,
     \frac{\sum_{i=1}^N\|\mathbf{x}_i-\bmu\|^{2\alpha}}{N 2^\alpha}=\frac{A_{2,d}(\kappa,\alpha,\alpha)}{A_{2,d}(\kappa,\alpha,0)}\right\}.
\end{equation*}
Consider the value under $\inf$ in $S_2(\mathcal{X}_N)$. Under the condition $ \frac{1}{N 2^\alpha}\sum_{i=1}^N\|\mathbf{x}_i-\bmu\|^{2\alpha}=\frac{A_{2,d}(\kappa,\alpha,\alpha)}{A_{2,d}(\kappa,\alpha,0)},$ we have 
$$-\log c_{2,d}(\kappa,\alpha) +\frac{\kappa }{\alpha}\frac{A_{2,d}(\kappa,\alpha,\alpha)}{A_{2,d}(\kappa,\alpha,0)}=-\left(\log c_{2,d}(\kappa,\alpha) -\frac{\kappa }{\alpha 2^{\alpha}}\frac{1}{N}\sum_{i=1}^N\|\mathbf{x}_i-\bmu\|^{2\alpha}\right)=-\frac{l_2(\mathcal{X}_N)}{N}.$$
Thus,  
$$S_2(\mathcal{X}_N)=-\frac{1}{N}\sup_{\substack{\alpha,\kappa>0,\\\bmu\in \Sp^{d-1}}}\left\{l(\mathcal{X}_N)\left| \,
     \frac{\sum_{i=1}^N\|\mathbf{x}_i-\bmu\|^{2\alpha}}{N 2^\alpha}=\frac{A_{2,d}(\kappa,\alpha,\alpha)}{A_{2,d}(\kappa,\alpha,0)}\right.\right\}.$$
Let us consider unconditional maximization of log-likelihood $l_2(\mathcal{X}_N)$. Partial derivative with respect to $\kappa$ equals
\begin{align*}
    \frac{\partial l_2(\mathcal{X}_N)}{\partial \kappa}&=\frac{\partial }{\partial \kappa}\left(\log\frac{\Gamma\left(\frac{d-1}{2}\right)}{2 \pi^{\frac{d-1}{2}}}-\log A_{2,d}(\kappa,\alpha,0)-\frac{\kappa }{\alpha 2^{\alpha}}\frac{1}{N}\sum_{i=1}^N\|\mathbf{x}_i-\bmu\|^{2\alpha}\right)\\
    &=\frac{1}{\alpha}\frac{A_{2,d}(\kappa,\alpha,\alpha)}{A_{2,d}(\kappa,\alpha,0)}-\frac{1}{\alpha 2^{\alpha}N}\sum_{i=1}^N\|\mathbf{x}_i-\bmu\|^{2\alpha},
\end{align*}
where we used $\frac{\partial }{\partial \kappa}A_{2,d}(\kappa,\alpha,0)=-\frac{1}{\alpha}A_{2,d}(\kappa,\alpha,\alpha).$ Thus, the superemum in $S_2(\mathcal{X}_N)$ with respect to $\kappa$ coincides with the unconditional supremum of $l_2(\mathcal{X}_N)$ and 
\begin{equation}
    S_2(\mathcal{X}_N)=-\frac{1}{N}\sup_{\substack{\alpha,\kappa>0,\\\bmu\in \Sp^{d-1}}}l_2(\mathcal{X}_N).
\end{equation}

Let $\Theta_0$ be a compact subset of $\R^2_+$ large enough to contain all values of parameters $(\alpha,\kappa)$ appearing in practice. Consider the following hypotheses  
\begin{itemize}
    \item  $H_{2,0}:$ $X\sim \mathrm{GvMF}_{2,d},$ for some $(\alpha,\kappa)\in \Theta_{0},$ 
    \item  $H_{2,1}:$ $X\not\sim \mathrm{GvMF}_{2,d}$ for all $(\alpha,\kappa)\in \Theta_{0}.$ 
\end{itemize}

Since $\Theta_0$ is compact, maximum likelihood estimators $\hat{\alpha}_{L},$ $\hat{\kappa}_{L}$ are consistent. We proved in Theorem \eqref{thm:L2} that the $k$th nearest neighbour estimator  $\widehat{H}_{N,k}$ of  $H(\mathbf{X})$  
is $L^2-$consistent for any $k\in \N.$ Thus, we test $H_{2,0}$ vs. $H_{2,1}$ with the statistic
\begin{align}
    \label{def:T2MLE}\hat{T}^{L}_{2,k}(\mathcal{X}_N):=-\log c_{2,d}(\hat{\kappa}_{L},\hat{\alpha}_{L}) +\frac{\hat{\kappa}_{L} }{\hat{\alpha}_{L}}\frac{A_{2,d}(\hat{\kappa}_{L},\hat{\alpha}_{L},\hat{\alpha}_{L})}{A_{2,d}(\hat{\kappa}_{L},\hat{\alpha}_{L},0)}-\widehat{H}_{N,k}
\end{align}
which tends in probability to 0, as $N\to \infty.$
We reject $H_{2,0}$ with level of significance $\beta$ if $|\hat{T}^{L}_{2,k}(\mathcal{X}_N)|\geq x_\beta,$ where  $x_\beta$ is a critical value determined by 
$  \pr_{H_0}(|\hat{T}^{L}_{2,k}(\mathcal{X}_N)|\geq x_\beta)\leq \beta.$

\begin{remark}
It is easy to see that maximum likelihood estimates of $\alpha,\kappa,$ estimator $\hat{H}_{N,k}$ and the statistics $\hat{T}^{L}_{2,k}$  are rotational invariant.

Actually, we can replace the maximum likelihood estimates of $\alpha,\kappa$ in \eqref{def:T2MLE} by any consistent estimates $\hat{\alpha},\hat{\kappa}.$ Indeed,
 if $\mathbf{x}_1\sim \mathrm{GvMF}_{2,d}(\alpha,\kappa,\bmu)$ under hypothesis $H_0,$ then  $$2^{\hat{\alpha}}\frac{A_{2,d}(\hat{\kappa},\hat{\alpha},\hat{\alpha})}{A_{2,d}(\hat{\kappa},\hat{\alpha},0)}\xrightarrow[N\to\infty]{P}2^{\alpha}\frac{A_{2,d}(\kappa,\alpha,\alpha)}{A_{2,d}(\kappa,\alpha,0)}=\E\|\mathbf{x}_1-\bmu\|^{\alpha}$$  and $$\hat{T}_{2,k}(\mathcal{X}_N):=-\log c_{2,d}(\hat{\kappa},\hat{\alpha}) +\frac{\hat{\kappa} }{\hat{\alpha}}\frac{A_{2,d}(\hat{\kappa},\hat{\alpha},\hat{\alpha})}{A_{2,d}(\hat{\kappa},\hat{\alpha},0)}-\widehat{H}_{N,k}\xrightarrow[\text{under }H_0]{P} H(\mathbf{x}_1)-H(\mathbf{x}_1)=0$$ as $N\to\infty.$

\end{remark}

The critical values $x_\beta$ can be found by Monte Carlo simulations of test statistics $\hat{T}^{L}_{2,N}$ or $\hat{T}_{2,N}.$

\subsection{Type I and axial data}
The goodness of fit test for the axial generilized von-Mises distribution and the distribution of the I-type  are constructed similarly to II-type distributions. 
Let $\Theta_0$ be a compact subset of $\R^2_+$ large enough to contain all values parameters $(\alpha,\kappa)$ appearing in practice. Let $j=1,3$ and  consider the following hypotheses  
\begin{itemize}
    \item  $H_{j,0}:$ $X\sim \mathrm{GvMF}_{j,d},$ for some $(\alpha,\kappa)\in \Theta_{0},$ 
    \item  $H_{j,1}:$ $X\not\sim \mathrm{GvMF}_{j,d}$ for all $(\alpha,\kappa)\in \Theta_{0}.$ 
\end{itemize}
For testing $H_{1,0}$ vs. $H_{1,1}$ we use the statistic $\hat{T}_{1,N}$ given by
\begin{align}
    \label{def:T1}\hat{T}_{1,k}(\mathcal{X}_N):=-\log c_{1,d}(\hat{\kappa},\hat{\alpha}) -\frac{\hat{\kappa} }{\hat{\alpha}}\frac{A_{1,d}(\hat{\kappa},\hat{\alpha},\hat{\alpha})-A_{1,d}(-\hat{\kappa},\hat{\alpha},\hat{\alpha})}{A_{1,d}(\hat{\kappa},\hat{\alpha},0)+A_{1,d}(-\hat{\kappa},\hat{\alpha},0)}-\widehat{H}_{N,k},
\end{align}
where $\hat{\alpha},\hat{\kappa}$ are some consistent estimates of $\alpha,\kappa.$

For the axial distribution, we test $H_{3,0}$ vs. $H_{3,1}$ by the test statistic $\hat{T}_{3,N}$ given by
\begin{align}
    \label{def:T3}\hat{T}_{3,k}(\mathcal{X}_N):=-\log c_{3,d}(\hat{\kappa},\hat{\alpha}) -\frac{\hat{\kappa} }{\hat{\alpha}}\frac{A_{1,d}(\hat{\kappa},\hat{\alpha},\hat{\alpha})}{A_{1,d}(\hat{\kappa},\hat{\alpha},0)}-\widehat{H}_{N,k},
\end{align}
where $\hat{\alpha},\hat{\kappa}$ are some consistent estimates of $\alpha,\kappa.$



\section{Numerical experiments}
\subsection{Simulation}
In this section, we provide the method for simulation of $\mathrm{GvMF}_{j,d}$ distributed random vectors and study the behaviour of the test statistic $\hat{T}_{j,k}$ on simulated samples $j=1,2,3.$

Let $X_j\sim \mathrm{GvMF}_{j,d}(\alpha, \kappa,\bmu),j=1,2,3.$  Due to the tangent-normal decomposition 
$$\mathbf{X}_j=(\bmu^{\T}\mathbf{X}_j)\bmu + \sqrt{1-(\bmu^{\T}\mathbf{X}_j)^2}\mathbf{Y_j},$$
where $\mathbf{Y_j},j=1,2,3$ are orthogonal to $\bmu$ and uniformly distributed on $\Sp^{d-2}.$ So, in order to simulate $\mathbf{X}_j$ we can easily simulate random vectors  $\mathbf{Y_j}$ and independent random variables $\bmu^{\T}\mathbf{X}_j.$

Let us find the distributions of $\bmu^{\T}\mathbf{X}_j$ $j=1,2,3.$
\begin{lemma} The random variables $\bmu^{\T}\mathbf{X}_1,\bmu^{\T}\mathbf{X}_2,$ and $\bmu^{\T}\mathbf{X}_3$ have probability densities $f_1,f_2,$ and $f_2$ respectively, given by 
\begin{align}
\label{f1}f_1(y)&=\frac{2 \pi^{\frac{d-1}{2}}c_{1,d}(\kappa,\alpha)}{\Gamma\left(\frac{d-1}{2}\right)}\exp\left(\frac{\kappa}{\alpha} y^{<\alpha>}\right)(1-y^2)^{\frac{d-3}{2}}, y\in[-1,1],\\
\label{f2} f_2(y)&=\frac{2 \pi^{\frac{d-1}{2}}c_{2,d}(\kappa,\alpha)}{\Gamma\left(\frac{d-1}{2}\right)}\exp\left(-\frac{\kappa}{\alpha}(1-y)^{\alpha}\right)(1-y^2)^{\frac{d-3}{2}}, y\in[-1,1],\\
\label{f3} f_3(y)&=\frac{2 \pi^{\frac{d-1}{2}}c_{3,d}(\kappa,\alpha)}{\Gamma\left(\frac{d-1}{2}\right)}\exp\left(\frac{\kappa}{\alpha} |y|^\alpha \right)(1-y^2)^{\frac{d-3}{2}}, y\in[-1,1].
\end{align}

\end{lemma}
\begin{proof}
Consider $\bmu^{\T}\mathbf{X}_1.$ It follows from  Lemma \ref{lmm1} that
\begin{align*}\pr(\bmu^{\T}\mathbf{X}_1\leq u)&=c_{1,d}(\kappa,\alpha)\int_{\|\mathbf{x}\|=1}\mathbbm{1}\{\bmu^{\T}\mathbf{x}\leq u\}\exp\left(\frac{\kappa}{\alpha} (\bmu^{\T}\mathbf{x})^{<\alpha>}\right) \sigma(d \mathbf{x})\\
=&c_{1,d}(\kappa,\alpha)\frac{2 \pi^{\frac{d-1}{2}}}{\Gamma\left(\frac{d-1}{2}\right)}\int_{-1}^1\mathbbm{1}\{y\leq u\}\exp\left(\frac{\kappa}{\alpha} (y)^{<\alpha>} \right)(1-y^2)^{\frac{d-3}{2}}d y=\int_{-1}^u f_1(y)dy.
\end{align*}
The cases of $\bmu^{\T}\mathbf{X}_2$ and $\bmu^{\T}\mathbf{X}_3$ are similar.
\end{proof}

Applying described procedure of simulation we obtain several samples of generalized von Mises-Fisher distributions with 1000 entries on 2-dimensional sphere. 
For all samples we fix mean direction $\bmu=(0,\sqrt{2}/2,\sqrt{2}/2).$ For different values of $\alpha$ and $\kappa$ we present locations of samples entries on a unit sphere: for Type I, see Figure \ref{fig:S:type1_a_05} (with $\alpha=0.5$) and Figure \ref{fig:type1_a_15_sphere} (with $\alpha=1.5$); for Type 2, Figures  \ref{fig:type2_a_05_sphere} and \ref{fig:type2_a_15_sphere} with $\alpha=0.5$ and $\alpha=1.5,$ respectively, and the samples of axial data are presented in Figures  \ref{fig:type3_a_05_sphere} (with $\alpha=0.5$) and \ref{fig:type3_a_15_sphere} (with $\alpha=1.5$).
The corresponding histograms and probability densities of random variables $\bmu^{\T}\mathbf{X}_i,i=1,2,3$ can be found in Figures \ref{fig:type1_a_05} ($\alpha=0.5$) and \ref{fig:type1_a_15} ($\alpha=1.5$) for Type I, in Figures \ref{fig:type2_a_05} ($\alpha=0.5$) and \ref{fig:type2_a_15} ($\alpha=1.5$) for Type II, and in Figures \ref{fig:type3_a_05} ($\alpha=0.5$) and \ref{fig:type3_a_15} ($\alpha=1.5$) for axial data.

One can observe that larger values of parameter $\kappa$ corresponds to more concentrated samples along direction $\bmu.$

\subsection{Parameter estimation}
We provide here a short computational study of the parameter estimation methods from sections \ref{MLE} and \ref{MME}. 
For each type of distributions $\mathrm{GvMF_{1,3}},\mathrm{GvMF_{2,3}},\mathrm{GvMF_{3,3}}$ we simulate 1000 samples with $N=1000$ entries each for several values of $\alpha\in\{0.5,1,1.5,2,2.5,3\}$ and $\kappa\in \{0.1,0.5,1,1.5,2,2.5,3,4,5,6,7\}.$ For each sample we compute maximum likelihood estimates  $\hat{\bmu}_{L},\hat{\alpha}_{L},\hat{\kappa}_{L}$ and moment estimates $\hat{\bmu}_{M},\hat{\alpha}_{M},\hat{\kappa}_{M}.$ We present the sample mean square errors of $\hat{\alpha}_{L}$ and $\hat{\alpha}_{M}$ in Tables \ref{MLET1n1000csm} (type I), \ref{alpha_hat_Type2} (type II), and \ref{alpha_type3} (axial type). The mean square errors of $\hat{\kappa}_{L}$ and $\hat{\kappa}_{M}$ can be found in Tables \ref{MLET2n1000csm} (Type I), \ref{kappa_hat_Type2} (Type II), and \ref{kappa_type3} (axial data).

We group error values of $\hat{\kappa}_{L},$ $\hat{\kappa}_{M}$ and  $\hat{\alpha}_{L},$ $\hat{\alpha}_{M}$ in order to decide which method is more appropriate for parameter estimation.  The method of moments is preferable for Types I and II, if the computing power plays a decisive role. For axial data this method has no such advantages because we need to operate with an orientation tensor. 

The experiments shows that the speed of convergence $\hat{\alpha} \to \alpha$ and $\hat{\kappa} \to \kappa$  depend on the values of $\alpha,$ and $\kappa$ and $\hat{\kappa} \to \kappa$  faster than $\hat{\alpha} \to \alpha.$ We can also conclude that the errors of maximum likelihood estimators are generally less than moment estimators. Comparing the errors by the distribution type, we observe that samples of Type II very often carry the smallest error.

\subsection{Entropy estimation}
In this section, we apply the $k$th nearest neighbour estimator \eqref{eq_of_H_N} to the simulated set of samples. We compute estimates $\widehat{H}_{N,k}(\mathcal{X}_N)$ and their sample variances $\mathrm{sVar}(\widehat{H}_{N,k}(\alpha,\kappa))$ for $k=1,2,3,4,5,$  $\kappa\in \{0.1,0.5,1,1.5,2,2.5,3,4,5,6,7\},$ and $\alpha\in\{0.5,1,1.5,2,2.5,3\}.$ 
In Figure \ref{fig:HistH}, we illustrate the distribution of $\widehat{H}_{N,k}(\mathcal{X}_N)$ with $k=3$ by histograms for samples simulated from distributions $\mathrm{GvMF}_{j,3}(1.5,2,\cdot),j=1,2,3.$ 
\begin{figure}[h]
     \centering
     \begin{subfigure}[b]{0.32\textwidth}
         \centering
         \includegraphics[width=\textwidth]{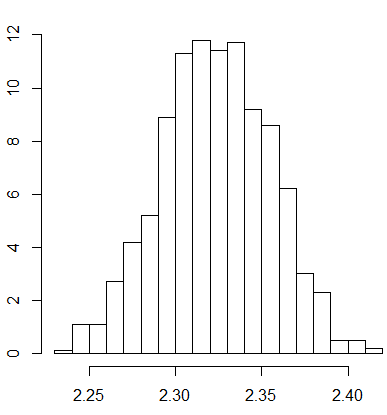}
         \caption{Type I}
     \end{subfigure}
     \hfill
     \begin{subfigure}[b]{0.32\textwidth}
         \centering
         \includegraphics[width=\textwidth]{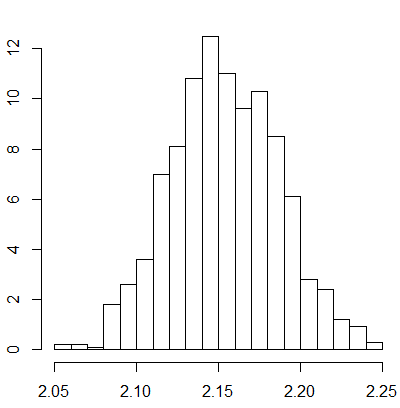}
         \caption{Type II}
     \end{subfigure}
     \hfill\begin{subfigure}[b]{0.32\textwidth}
         \centering
         \includegraphics[width=\textwidth]{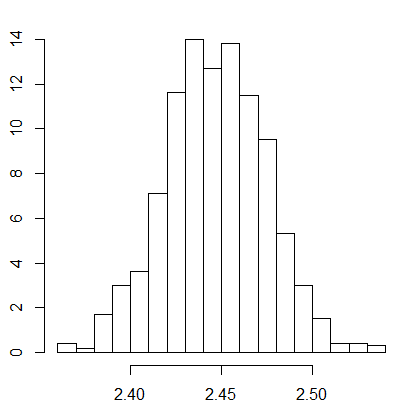}
         \caption{Axial type}
     \end{subfigure}
     \caption{Histograms of $\hat{H}_{N,3}(\mathcal{X}_N),$ $\mathcal{X}_N\sim \mathrm{GvMF}_{j,3}(\alpha,\kappa,\cdot)$ with $\alpha=1.5$ and $\kappa=2.$}
    \label{fig:HistH}
\end{figure}

In order to choose the right value of $k$ we compare the sample variances $\mathrm{sVar}(\widehat{H}_{N,k}(\alpha,\kappa))$ for $k=1,2,3,4,5.$ The minimum and maximum values of $\mathrm{sVar}(\widehat{H}_{N,k}(\alpha,\kappa))$ with respect to $\alpha$ and $\kappa$ are presented in Table \ref{varE}. Our results confirm the conclusion in \citep{berrett2019}, that is, the  asymptotic variance of $\widehat{H}_{N,k}$ decreases rapidly up to $k=3.$  One can observe that $k-$nearest neighbour estimates  depend on values $\alpha$ and $\kappa.$ Although,  the sample variances are quite small for all examined values of $\alpha$ and $\kappa$ and sample size $N=1000.$

\begin{table}[h]
\footnotesize
    \centering
    \begin{tabular}{lrccccc}
Distribution    & & $k=1$ &$k=2$ &$k=3$ &$k=4$   &$k=5$  \\ 
        \hline
$\mathrm{GvMF}_{1,3}(\alpha,\kappa)$ &$\min_{\alpha, \kappa}$(sVar) &0.00214	&0.00092	&0.00058	&0.00042 &0.00034\\
&$\max_{\alpha, \kappa}$(sVar)   & 0.00388 &	0.00239	& 0.00208 &	0.00185	& 0.00152
\\
\hline
$\mathrm{GvMF}_{2,3}(\alpha,\kappa)$&$\min_{\alpha, \kappa}$(sVar) &0.00212 &	0.00093	& 0.00059	& 0.00043 &	0.00034\\
&$\max_{\alpha, \kappa}$(sVar)   & 0.00404	&0.00304	&0.00275	&0.00152	&0.00257
\\
\hline
$\mathrm{GvMF}_{3,3}(\alpha,\kappa)$&$\min_{\alpha, \kappa}$(sVar) &0.00212	& 0.00093	& 0.00059	& 0.00043	&0.00034\\
&$\max_{\alpha, \kappa}$(sVar)  & 0.00334 & 0.00207	& 0.00170	& 0.00152 &	0.00144\\
\hline
    \end{tabular}
            \caption{Sample variance sVar$(\widehat{H}_{N,k}(\alpha,\kappa))$ for distribution of type I}
    \label{varE}
\end{table}
Thus, we choose $k=3$ for computations in the next sections.



\subsection{Test statistic}
In this section, we present our study of the goodness of fit tests from Section \ref{Sec:GFtest} and their test statistics $\hat{T}_{j,k}(\mathcal{X}_N),j=1,2,3$ from \eqref{def:T2MLE}, \eqref{def:T1} and \eqref{def:T3}  with $k=3$ and $N=1000.$ We compute $\hat{T}^{L}_{j,k}(\mathcal{X}_N)$ and $\hat{T}_{j,k}^{M}(\mathcal{X}_N)$  separately with maximum likelihood estimates and estimates be the method of moments of parameters $\alpha,\kappa,\bmu,$ respectively.

For comparison of different types of estimates, we look on the sample variances $\mathrm{sVar}(\hat{T}^{M}_{j,3}(\alpha,\kappa))$ and $\mathrm{sVar}(\hat{T}^{L}_{j,3}(\alpha,\kappa))$ of $\hat{T}^{M}_{j,3}(\mathcal{X}_N)$ and $\hat{T}_{j,3}^{L}(\mathcal{X}_N),$ respectively, for all combinations of parameters $\kappa\in \{0.1,0.5,1,1.5,2,2.5,3,4,5,6,7\}$ and $\alpha\in\{0.5,1,1.5,2,2.5,3\}.$ The minimum and maximum values of $\mathrm{sVar}(\hat{T}^{M}_{j,3}(\alpha,\kappa))$ and $\mathrm{sVar}(\hat{T}^{L}_{j,3}(\alpha,\kappa))$ are presented in Table \ref{varT}. Numbers in this tables demonstrate significant benefits  of the maximum likelihood method over the method of moments for distributions of I and II types. For axial data, one can also prefer $\hat{T}^{L}_{j,k}.$ Additionally, one can observe from Table \ref{varT} and tables with errors of estimates $\hat{\kappa}^{L},\hat{\kappa}^{M},\hat{\alpha}^{L},\hat{\alpha}^{M}$ that the statistics $\hat{T}^{M}_{j,k}$ and  $\hat{T}^{L}_{j,k}$ are much more accurate than estimators of parameters and they have small variances even for small $\alpha,\kappa$ in contrast to $\hat{\alpha},\hat{\kappa},$ whose deviations are large. 
\begin{table}[]
\footnotesize
    \centering

    \begin{tabular}{l|ccc|ccc|}
    Type of estimates&\multicolumn{3}{c|}{Moments} & \multicolumn{3}{c}{Maximum likelihood}\\
    \hline
    Distribution& $\mathrm{GvMF}_{1,3}$ &$\mathrm{GvMF}_{2,3}$ &$\mathrm{GvMF}_{3,3}$ & $\mathrm{GvMF}_{1,3}$ & $\mathrm{GvMF}_{2,3}$ &$\mathrm{GvMF}_{3,3}$  \\ 
        \hline 
$\min_{\alpha, \kappa}$(sVar) &0.000563	&0.000538&	0.000573	&	0.000558 &	0.000535&	0.000544\\
$\max_{\alpha, \kappa}$(sVar)   & 0.001014 &	0.001558 &	0.000734 &		0.000665	& 0.000688	& 0.000667
\\
\hline
    \end{tabular}
            \caption{Sample variances sVar$(\hat{T}_{j,3}(\alpha,\kappa))$ }
    \label{varT}
\end{table}

In Section \ref{Sec:GFtest}, we choose the two-sided test with rejection criteria $|\hat{T}^{L}_{j,k}|>x_{\beta,j}.$ We confirm this choice by  histograms $\hat{T}^{L}_{j,3}$ with $\alpha=1.5$ and $\kappa=2,$ see Figure \ref{fig:HistT}. 

\begin{figure}[h]
     \centering
     \begin{subfigure}[b]{0.32\textwidth}
         \centering
         \includegraphics[width=\textwidth]{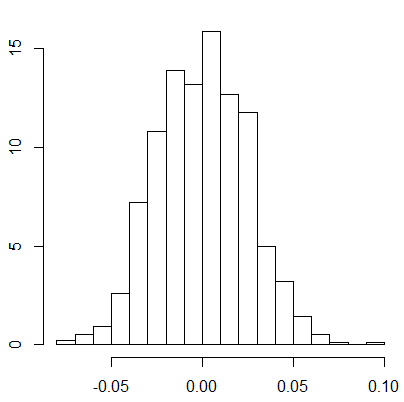}
         \caption{$\mathrm{GvMF}_{1,3}$, sVar=0.0006019}
     \end{subfigure}
     \hfill
     \begin{subfigure}[b]{0.32\textwidth}
         \centering
         \includegraphics[width=\textwidth]{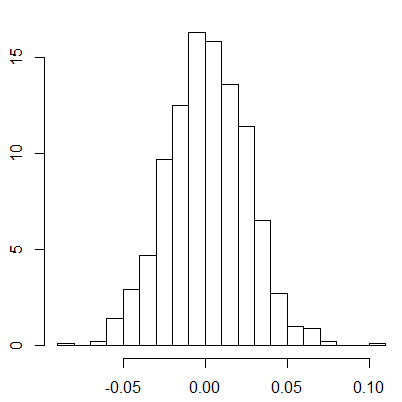}
         \caption{$\mathrm{GvMF}_{2,3}$, sVar=0.0005924}
     \end{subfigure}
     \hfill\begin{subfigure}[b]{0.32\textwidth}
         \centering
         \includegraphics[width=\textwidth]{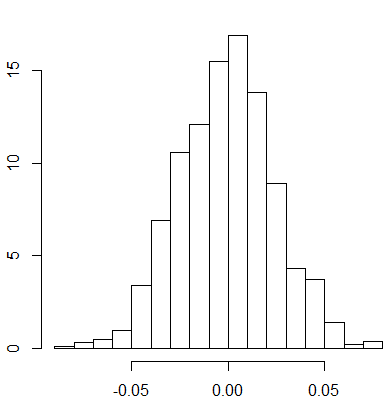}
         \caption{$\mathrm{GvMF}_{3,3}$, sVar=0.0006118}
     \end{subfigure}
     \caption{Histograms of $\hat{T}^{L}_{j,3}(\mathcal{X}_N),\mathcal{X}_N\sim \mathrm{GvMF}_{j,3}(\alpha,\kappa,\cdot)$ with $\alpha=1.5$ and $\kappa=2.$}
           \label{fig:HistT}
\end{figure}

We see that the statistics $\hat{T}^{L}_{j,3}$ have approximately symmetric distribution with mode at 0. Therefore, we put rejection region as $(-\infty,-x_\beta)\cup[x_\beta,+\infty),$ where critical values $x_{\beta}$ are obtained as a samples quantiles $\pr_{H_0}(|\hat{T}_{i,3}|>x_{\beta,j})\leq \beta,j=1,2,3$ The corresponding values of $x_{\beta,j}$ with significance level $\beta=0.05$ are presented in Table \ref{tab:T1TstatMLEn1000:5} for $\hat{T}^{L}_{1,3}$, in Table \ref{tab:T2TstatMLEn1000:5} for $\hat{T}^{L}_{2,3}$, and in Table \ref{tab:T3TstatMLEn1000:5} for $\hat{T}^{L}_{3,3}.$
\begin{table}[htb!]
  
\begin{tabular}{cccccccccccc}
 & {0.1}&{0.5}&{1} &{1.5}&{2}&{2.5}& {3}&{4}&{5}&{6}&{7} \\
0.5& 4.884& 4.626& 5.128& 5.281& 5.069& 4.960& 4.726& 4.999& 5.132& 5.360& 5.301 \\
1  & 4.727& 4.792& 4.879& 4.736& 4.757& 4.843& 4.920& 5.217& 4.951& 5.278& 5.091 \\
1.5& 4.935& 4.914& 4.657& 4.731& 4.745& 4.879& 4.849& 4.995& 5.152& 4.990& 5.009 \\
2  & 4.916& 4.873& 4.839& 4.982& 4.801& 4.987& 4.752& 4.934& 4.829& 5.076& 5.010 \\
2.5& 4.916& 4.921& 5.276& 4.932& 4.854& 4.852& 4.711& 4.700& 4.860& 4.983& 5.373 \\
3  & 4.687& 4.821& 4.783& 4.626& 4.926& 4.704& 4.656& 4.690& 4.631& 4.706& 4.844 
\end{tabular}
\caption{Critical values $x_{\beta,1}$ for test statistic $\hat{T}^L_{1,3}$ and $\beta = 0.05,$ with respect to $\alpha$ (rows) and $\kappa$ (columns), multiplied by $10^2.$}
  \label{tab:T1TstatMLEn1000:5}
\end{table}

\begin{table}[htb!]

\begin{tabular}{cccccccccccc}
 & {0.1}&{0.5}&{1} &{1.5}&{2}&{2.5}& {3}&{4}&{5}&{6}&{7}\\
0.5 & 4.921 &  4.795 &  4.750 &  4.962 &  4.711 &  5.015 &  4.742 &  5.182 &  5.430 &  5.374 &  5.388 \\
1   &  4.718 &  4.688 &  4.858 &  4.937 &  4.996 &  4.829 &  4.870 &  5.065 &  5.234 &  5.287 &  5.449 \\
1.5 &  4.698 &  4.916 &  4.943 &  4.714 &  4.824 &  4.988 &  4.964 &  4.613 &  4.975 &  5.267 &  5.263 \\
2   &  4.920 &  4.989 &  4.852 &  4.611 &  4.890 &  4.869 &  5.147 &  5.037 &  4.941 &  5.423 &  5.217 \\
2.5 &  5.066 &  4.753 &  5.034 &  4.805 &  4.768 &  4.930 &  4.973 &  5.346 &  5.283 &  5.138 &  5.074 \\
3   &  4.713 &  4.417 &  4.731 &  4.907 &  5.001 &  5.007 &  4.870 &  5.023 &  5.015 &  5.083 &  5.116
\end{tabular}
\caption{Critical values $x_{\beta,2}$ for test statistic $\hat{T}^L_{2,3}$ and $\beta = 0.05,$ with respect to $\alpha$ (rows) and $\kappa$ (columns), multiplied by $10^2.$}
\label{tab:T2TstatMLEn1000:5}
\end{table}


\begin{table}[htb!]
\begin{tabular}{cccccccccccc}
 & {0.1}&{0.5}&{1} &{1.5}&{2}&{2.5}& {3}&{4}&{5}&{6}&{7}\\
0.5 &  4.932 &  5.095 &  4.843 &  4.863 &  5.040 &  4.995 &  5.000 &  5.477 &  5.810 &  5.219 &  5.527 \\
1   &  4.956 &  4.793 &  4.636 &  4.912 &  4.896 &  4.873 &  4.959 &  5.251 &  5.196 &  5.565 &  5.917 \\
1.5 &  4.620 &  4.873 &  4.781 &  4.892 &  4.839 &  5.006 &  4.806 &  4.824 &  5.016 &  5.239 &  5.370 \\
2   &  4.727 &  5.049 &  4.804 &  4.567 &  4.811 &  4.651 &  4.842 &  4.999 &  4.836 &  5.162 &  5.191 \\
2.5 &  4.925 &  4.799 &  4.938 &  4.831 &  4.735 &  4.834 &  5.024 &  4.739 &  4.917 &  4.899 &  4.703 \\
3   &  4.960 &  4.901 &  4.904 &  4.649 &  5.037 &  4.843 &  4.898 &  4.852 &  4.828 &  4.895 & 4.974
\end{tabular}
\caption{Critical values $x_{\beta,3}$ for test statistic $\hat{T}^L_{3,3}$ and $\beta = 0.05,$ with respect to $\alpha$ (rows) and $\kappa$ (columns), multiplied by $10^2.$}
  \label{tab:T3TstatMLEn1000:5} 
\end{table}

We provide also the study of the goodness of fit test's power for the samples from Fisher-Bingham distribution. We put $\bmu_1=(1,0,0)^{\T}$ and $\bmu_2=(0,\sqrt{2}/2,\sqrt{2}/2)^{\T}$ and consider the following series of hypotheses. 

For type I:
\begin{itemize}
    \item $H^1_0:$ $\mathbf{X}\sim \mathrm{GvMF}_{1,3},$
    \item $H^1_{1,j}:$ $\mathbf{X}$ has the Fisher-Bingham distribution with density $\propto\exp(3 \bmu_1^{\T}\mathbf{x}+0.35j(\bmu_2^{\T}\mathbf{x})^2),$ $\mathbf{x}\in\Sp^{2},$ $j=1,\ldots,20.$
\end{itemize}

For axial type:
\begin{itemize}
    \item $H^2_0:$ $\mathbf{X}\sim \mathrm{GvMF}_{3,3},$
    \item $H^2_{1,j}:$ $\mathbf{X}$ has the Fisher-Bingham distribution with density  $\propto\exp(0.05j (\bmu_1^{\T}\mathbf{x})+6(\bmu_2^{\T}\mathbf{x})^2),$ $\mathbf{x}\in\Sp^{2},$ $j=1,\ldots,20.$
\end{itemize}
For each $j=1,\ldots,20,$ we simulate 500 samples $\mathcal{X}^1_{j,N}$ under $H^1_{1,j}$ and  $\mathcal{X}^2_{j,N}$ under $H^2_{1,j}$  with sample size $N=1000.$ We use the simulation procedure from the $\mathrm{R}$ package `Directional' \citep{DirPack}. In order to simplify commutations, we reject $H^1_0$ and $H^2_0$ if  $|\hat{T}^{L}_{1,3}(\mathcal{X}^1_{j,N})|>x^{(1)}_\beta$ and  $|\hat{T}^{L}_{2,3}(\mathcal{X}^2_{j,N})|>x^{(2)}_\beta$ respectively, where critical values  $x^{(1)}_\beta=0.05373$ and $x^{(2)}_\beta=0.05917$ are taken as maximum of $x_\beta$ from Tables \ref{tab:T1TstatMLEn1000:5} and \ref{tab:T3TstatMLEn1000:5} for significance level $\beta=0.05.$

The ratios of rejections $H^1_0$ and $H^2_0$ are presented in Figure \ref{Pvalue}.
\begin{figure}[h]
     \centering
     \begin{subfigure}[b]{0.48\textwidth}
         \centering
         \includegraphics[width=\textwidth]{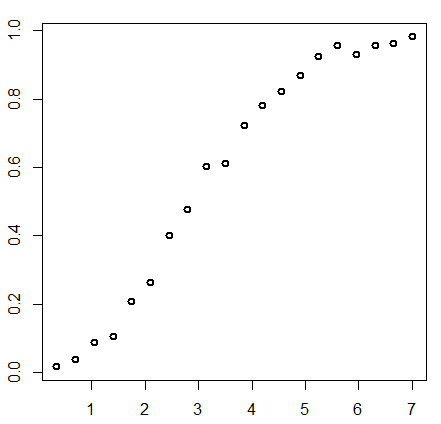}
         \caption{Type I, $0.35j$ on the $x-$axis. }
         \label{Pvalue:T1}
     \end{subfigure}
     \hfill
     \begin{subfigure}[b]{0.48\textwidth}
         \centering
         \includegraphics[width=\textwidth]{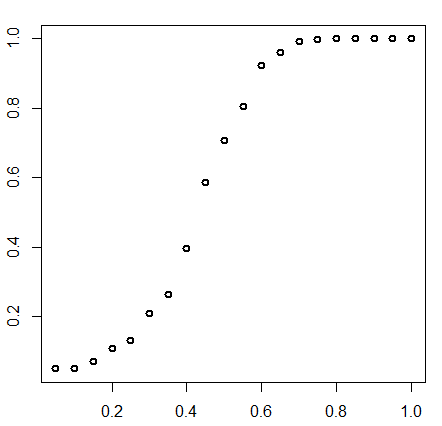}
         \caption{Axial type, $0.05j$ on the $x-$axis.}
         \label{Pvalue:T3}
     \end{subfigure}
     \caption{Powers of goodness of fit tests $H^1_0$ vs $H^1_{1,j}$ (left) and $H^2_0$ vs $H^2_{1,j}$ (right), $j=1,\ldots,20.$}
     \label{Pvalue}
\end{figure}

\section{Application to a real data set}
In this section, we apply the introduced goodness of fit test to the data set consists of fiber directions in a glass fibre reinforced composite material. The 3D-images  of a fibre composite obtained by
micro computed tomography and are provided by the Institute for Composite Materials (IVW) in Kaiserslautern, Germany, see Fig.~\ref{alphakappa} (left). The detailed description of the material can be found in \citep{IVWFibres} and it was the object of studies in \citep{Dresvyanskiy_2019} and \citep{dresvyanskiy2020detecting}, where  the regions of anomaly behaviour of the fibres were found.   The data set is provided by Prof. Claudia Redenbach (TU Kaiserslautern) and consists of local direction of fibres estimated by the tools of MAVI software \citep{MAVI}. Each data set entry ${Y}_{k},{k}=([1,97]\times[1,80]\times[1,64])\cap \N^3$ is the average of fibre local directions in small observation windows $\tilde{W}$ with $75\times75\times75$ voxels each. Note that some of such windows can be empty or they might contain not enough material for direction computation. 
We denote by $J_W$ the collection of indexes ${k}$ such that ${Y}_{k}$ is non-empty. In the considered data set $| J_W| =430741$ and its precise construction is in \citep{Dresvyanskiy_2019}. 

\begin{figure}
    \centering
   \begin{subfigure}[t]{0.45\textwidth}
    \includegraphics[width=0.8\linewidth]{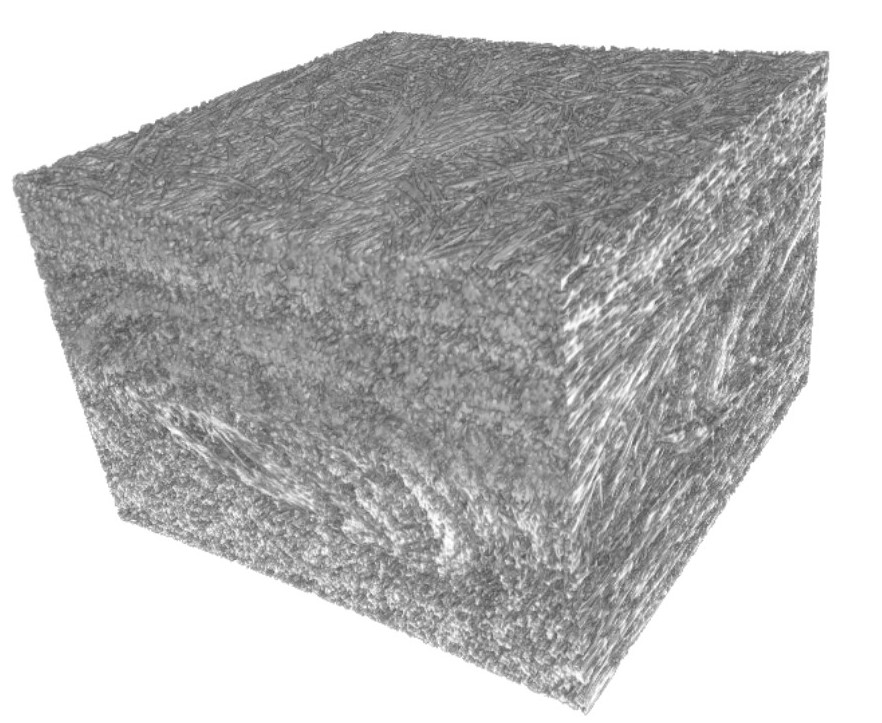}
    \caption{3D image of a glass fibre reinforced composite material.
970$\times$1469$\times$1217 voxels, spacing: 4$\mu$m.}
    \label{fig:materail}
    \end{subfigure}
    \begin{subfigure}[t]{0.45\textwidth}
    \includegraphics[width=\linewidth]{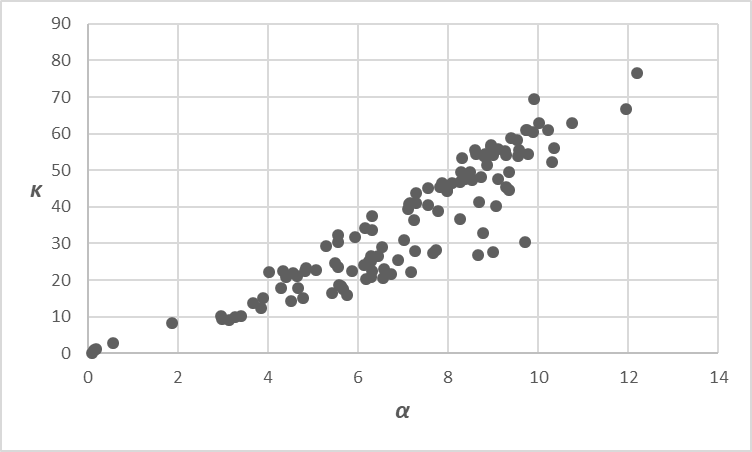}
    \caption{Maximum likelihood estimates of $\alpha$ and $\kappa$ for each subsample $\mathcal{X}_{\mathbf{l}}$ }
    \label{alphakappa}
    \end{subfigure}
    \caption{Testing on glass fibre reinforced composite material.}
\end{figure}

The estimating procedure of directions in MAVI software produces vectors on a unit sphere which are not necessarily symmetrically distributed. However, the fibres are not oriented, therefore we expect an axial distribution of their directions. We propose the symmetrization of original sample by 
$\mathcal{X}_N=\{\mathbf{X}_{\mathbf{k}}=\mathbf{Y}_{\mathbf{k}}\xi_\mathbf{k},\mathbf{k}\in J_W\},$ where $\xi_\mathbf{k},\mathbf{k}\in J_W$ are i.i.d random variables with $\pr(\xi_\mathbf{k}=+1)=\pr(\xi_\mathbf{k}=-1)=\frac{1}{2}.$
We separate the whole material into blocks $W_{\mathbf{l}},$ each of size $16\times15\times16,$ such that 
$J_\mathbf{l}=J_W\cap[l_1,l_1+16)\times[l_2,l_2+15)\times[l_3,l_2+16)$ and consider subsamples $\mathcal{X}_{\mathbf{l}}=\{\mathbf{X}_{\mathbf{k}},\mathbf{k}\in J_\mathbf{l}\}$ with simple sizes $2736\leq |\mathcal{X}_{\mathbf{l}}|\leq 3745.$ 
For each subsample $\mathcal{X}_{\mathbf{l}}$ we provide the introduced goodness of fit test for distributions $\mathrm{GvMF_{3,3}},$ i.e., we test
\begin{itemize}
    \item $H_{0,\mathbf{l}}:$ $\mathcal{X}_{\mathbf{l}}\sim \mathrm{GvMF_{3,3}}$
    \item $H_{1,\mathbf{l}}:$ $\mathcal{X}_{\mathbf{l}}\not\sim \mathrm{GvMF_{3,3}}.$
\end{itemize}

At first, we provide maximum likelihood estimation of parameters $\alpha$ and $\kappa$ (for the variety of their values $\hat{\alpha_{\mathbf{l}}}$ and $\hat{\kappa_{\mathbf{l}}}$ see Figure \ref{alphakappa}). Second, we need to simulate the samples of statistics $\hat{T}_{3,k}(\mathcal{X}_\mathbf{l})$ under hypotheses $\mathcal{X}_\mathbf{l}\sim \mathrm{GvMF}_{3,3}(\hat{\alpha_{\mathbf{l}}},\hat{\kappa_{\mathbf{l}}},\cdot)$ based on samples sizes $|\mathcal{X}_\mathbf{l}|.$ Unfortunately, our computational resources was limited and we have to group simulations with close values of $\hat{\alpha_{\mathbf{l}}}$ and $\hat{\kappa_{\mathbf{l}}}.$
One can observe that the majority of $\hat{\alpha_{\mathbf{l}}}$ belongs to the interval $[3,11]$ and the ratios $\frac{\hat{\kappa_{\mathbf{l}}}}{\hat{\alpha_{\mathbf{l}}}}$ are mostly in $[3,7].$ Therefore, we simulate 800 samples $\mathcal{Y}_N\sim \mathrm{GvMF}_{3,3}(\alpha,\kappa,\cdot)$ each of size $N=3500$ for all combinations of $\alpha\in \{4,6,8,10\}$ and $\frac{\kappa}{\alpha}\in \{4,6\}$ to obtain the corresponding empirical distributions of $\hat{T}_{3,3}(\mathcal{Y}_N).$

Then we compute statistics $\hat{T}_{3,3}(\mathcal{X}_\mathbf{l})$ for each $\mathbf{l}$ and their $p$-values. 
We obtain that our goodness of fit test rejects almost all hypotheses $H_{0,\mathbf{l}}$ with significance level $0.05,$ and detects 3 regions with directional distributions $\mathrm{GvMF_{3,3}}$ (see Table \ref{pvalue} for the samples $\mathcal{X}_{\mathbf{l}}$ with $p-$values greater than $0.01$). In order to illustrate how tight the fitted distributions are, we present for two blocks $W_{\mathbf{l}}$ with $\mathbf{l}=(49,61,1)$ and $\mathbf{l}=(49,61,1),$ the QQ-plots for samples $\{\hat{\bmu}_{\mathbf{l}}^{\T}\mathbf{X}_\mathbf{k}\}$ and distribution $f_3$ defined in \eqref{f3} with parameters $(\hat{\alpha_{\mathbf{l}}},\hat{\kappa_{\mathbf{l}}},\cdot),$ see Figure \ref{qq}.

\begin{table}[h]
\footnotesize
\centering
\begin{tabular}{llll|crrrrr}
$l_1$ & $l_2$   &  $l_3$ & $|\mathcal{X}_\mathbf{l}|$ & $\hat{\bmu}$& $\hat{\alpha}$& $\hat{\kappa}$& $\hat{H}_{N,3}$  & $\hat{T}_{3,3}(\mathcal{X}_\mathbf{l})$ & $p-$value     \\
\hline
49 & 46 & 1 & 3434 & (-0.0028,  0.9999,  -0.0135) & 8.80  & 53.90 & 0.4369 & 0.02344 & 0.0775\\
49 & 61 & 1 & 3222 & (-0.0149,  0.9998,  0.0091)  & 8.53  & 47.62 & 0.7132 & 0.01976 & 0.1234\\
49 & 16 & 17 & 3474 & (0.0175,   0.9977,  0.0657)  & 8.84  & 53.63 & 0.4690 & 0.02987 & 0.0263\\
49 & 61 & 17 & 3364 & (-0.0125,  0.9998,  -0.0152) & 10.22 & 60.91 & 0.4628 & 0.02946 & 0.0263\\
1  & 46 & 65 & 3319 & (-0.0297,  0.9986,  -0.0432) & 7.25  & 36.45 & 1.0455 & 0.02057 & 0.1275\\
\hline
\end{tabular}
\caption{Results of goodness of fit tests $H_{0,\mathbf{l}}$ vs. $H_{1,\mathbf{l}}$ for fiber directions in glass fibre reinforced composite material.}
\label{pvalue}
\end{table}

\begin{figure}[h]
    \centering
    \begin{subfigure}[b]{0.45\textwidth}
         \centering
         \includegraphics[width=\textwidth]{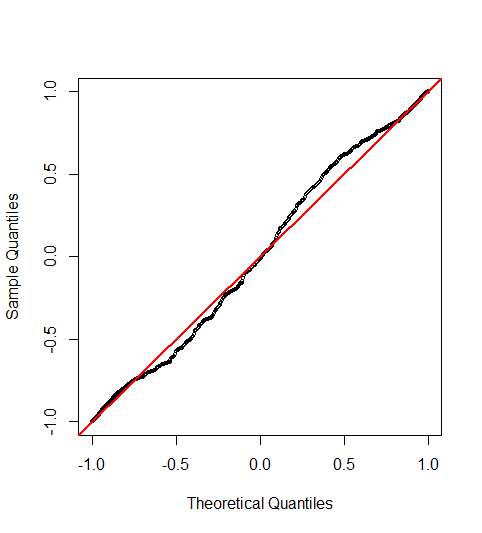}
         \caption{$\mathbf{l}=(49,61,1),\hat{\alpha_{\mathbf{l}}}=8.53,\hat{\kappa_{\mathbf{l}}}=47.62$}
     \end{subfigure}
     \begin{subfigure}[b]{0.45\textwidth}
         \centering
         \includegraphics[width=\textwidth]{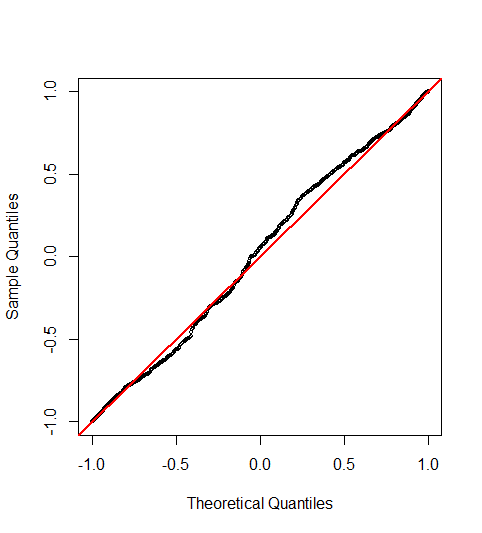}
         \caption{$\mathbf{l}=(1,46,65),\hat{\alpha_{\mathbf{l}}}=7.25,\hat{\kappa_{\mathbf{l}}}=36.45$}
     \end{subfigure}
    \caption{QQ plots for samples $\hat{\bmu}_{\mathbf{l}}^{\T}\mathbf{X}_\mathbf{l}$ and distributions with density $f_3$ \eqref{f3}.}
    \label{qq}
\end{figure}
\section*{Acknowledgement}
The research was partially supported by DFG Grant 390879134. The authors are grateful for Prof. Katja Schladitz for the help with the real data, Martin Gurka and Sebastian Nissle (Institut f\"{u}r Verbundwerkstoffe, Kaiserslautern) for permission to reuse the tomographic images, Prof. Claudia Redenbach for providing the data set of fiber directions.

\bibliographystyle{biometrika}
\bibliography{Lit} 

\newpage
\section*{Appendix}
Here we present tables of mean square errors of estimates of $\kappa$ and $\alpha$ and plots with realisations of the generalized von Mises-Fisher distributions


\begin{center}
\begin{table}[h]
\footnotesize
{
\begin{tabular}{cccccccccccc}
 & ${0.1}$ & ${0.5}$ & ${1}$  & ${1.5}$ & ${2}$ & ${2.5}$& ${3}$ & ${4}$ & ${5}$ & ${6}$ & ${7}$ \\[0.5ex] 
${0.5}$ & $2.03233 \atop 1.74912$ & $0.01131 \atop 0.01103$ & $0.00309 \atop 0.00294$ & $0.00229 \atop 0.00209$ & $0.00334\atop 0.00290$ & $0.00720 \atop 0.00600$ & $0.00234 \atop 0.01028$ & $0.00067 \atop 0.01568 $ & $ 0.00044 \atop 0.02000 $ & $0.00046 \atop 0.02222$ & $0.00040 \atop 0.02375$ 
\\[0.5ex] 
${1}$ &  $12.64360 \atop 15.83364$ & $0.27260 \atop 0.21929$ & $0.03760 \atop 0.03207$ & $0.01778 \atop 0.01518$ & $0.01126 \atop 0.00838$ & $0.00766 \atop 0.00569$ & $0.00633 \atop 0.00480$ & $0.00686 \atop 0.00508 $ & $0.00950 \atop 0.00644$ & $0.02210\atop 0.00850$ & $0.04043 \atop 0.01633$ 
\\[0.5ex] 
${1.5}$ &  $13.64093 \atop 24.85928$ & $2.89363 \atop 2.75421$ & $0.44954 \atop 0.20300$ & $0.12728 \atop 0.07216$ & $0.07166 \atop 0.03613$ & $0.04616 \atop 0.02355$ & $0.03330 \atop 0.01725$ & $0.02049 \atop 0.01081 $ & $0.01535 \atop 0.00904$ & $0.01465 \atop 0.00793$ & $0.01516 \atop 0.00854$ \\[0.5ex] 
${2}$ &  $12.95277 \atop 26.53262 $ & $8.34069 \atop 10.16847$ & $2.90354 \atop 1.30990$ & $0.89834 \atop 0.30051$ & $0.44383\atop 0.13914$ & $0.27728 \atop 0.08601$ & $0.16970 \atop 0.05726$ & $0.08593 \atop 0.03554 $ & $0.05445 \atop 0.02187$ & $0.04011 \atop 0.01650$ & $0.03321 \atop 0.01406$ \\[0.5ex] 
${2.5}$ &  $11.71910 \atop 26.53403$ & $8.68665 \atop 15.44292$ & $6.79656 \atop 4.46642$ & $3.96727 \atop 1.21787$ & $2.18844 \atop 0.42840$ & $1.23084 \atop 0.23206$ & $0.62982 \atop 0.16395$ & $ 0.35996 \atop 0.08478$ & $0.18528 \atop 0.05083 $ & $0.13112 \atop 0.03633$ & $0.09397 \atop 0.02807$ \\[0.5ex] 
${3}$ &  $11.23635 \atop 23.03320$ & $10.18631 \atop 18.71290$ & $8.74483 \atop 8.33385$ & $6.31636 \atop 3.35711$ & $ 6.33024 \atop 1.24416$ & $3.48461 \atop 0.60680$ & $2.84763 \atop 0.36062$ & $1.35829 \atop 0.20265 $ & $0.62283 \atop 0.11569 $ & $0.50274 \atop 0.08699$ & $0.33681\atop 0.05869 $ \\[0.5ex] 
\end{tabular}
}
\caption{Mean square errors of $\hat{\alpha}_{M}$ (top raw) and $\hat{\alpha}_{L}$ (bottom raw) for $\mathrm{GvMF}_{1,3}$ distribution, for different values of $\alpha$ (rows) and $\kappa$ (columns)}
\label{MLET1n1000csm}
\end{table}
\end{center}



\begin{center}
\begin{table}[h]
{
\footnotesize
\begin{tabular}{cccccccccccc}
 & ${0.1}$ & ${0.5}$ & ${1}$  & ${1.5}$ & ${2}$ & ${2.5}$& ${3}$ & ${4}$ & ${5}$ & ${6}$ & ${7}$ \\[0.5ex] 
${0.5}$ & $ 0.71826 \atop 0.37326 $ & $ 0.01899 \atop 0.01717 $ & $ 0.01699 \atop 0.01598 $ & $ 0.01898 \atop 0.01727 $ & $ 0.02602 \atop 0.02396 $ & $ 0.03563 \atop 0.03306 $ & $ 0.03419 \atop 0.04987 $ & $ 0.03749 \atop 0.07889 $ & $ 0.04394 \atop 0.10651 $ & $ 0.06683 \atop 0.13889 $ & $ 0.07928 \atop 0.17365 $ \\[0.5ex] 
${1}$ & $1.94226 \atop 2.46363 $ & $0.22319 \atop 0.11194 $ & $ 0.07076 \atop 0.05895 $ & $ 0.06534 \atop 0.05552 $ & $ 0.06965 \atop 0.05215 $ & $ 0.06449 \atop 0.05029 $ & $ 0.06737 \atop 0.05277 $ & $ 0.08770 \atop 0.06947 $ & $ 0.11201 \atop 0.08617 $ & $ 0.19431 \atop 0.10606 $ & $ 0.31639 \atop 0.16075 $ \\[0.5ex] 
${1.5}$ &  $1.53572 \atop 3.63032 $ & $ 2.11049 \atop 1.19008 $ & $ 0.64790 \atop 0.19097 $ & $ 0.26637 \atop 0.14496 $ & $ 0.25165 \atop 0.12205 $ & $ 0.23415 \atop 0.11508 $ & $ 0.22145 \atop 0.11355 $ & $ 0.20812 \atop 0.10874 $ & $ 0.20797 \atop 0.12384 $ & $ 0.23012 \atop 0.12786 $ & $ 0.24837 \atop 0.14523 $ \\[0.5ex] 
${2}$ &  $1.58362 \atop 3.87311 $ & $ 3.68444 \atop 3.19240 $ & $ 4.15468 \atop 0.92650 $ & $ 1.72159 \atop 0.39620 $ & $ 1.18887 \atop 0.29374 $ & $ 1.05376 \atop 0.27322 $ & $ 0.82951 \atop 0.24016 $ & $ 0.63399 \atop 0.24388 $ & $ 0.54231 \atop 0.20779 $ & $ 0.49912 \atop 0.19860 $ & $ 0.52341 \atop 0.21512 $ \\[0.5ex] 
${2.5}$ &  $1.34224 \atop 4.13401 $ & $ 2.17772 \atop 4.02490 $ & $ 6.54718 \atop 2.52365 $ & $ 7.27056 \atop 1.19600 $ & $ 5.81158 \atop 0.61967 $ & $ 4.41608 \atop 0.50278 $ & $ 2.47149 \atop 0.50154 $ & $ 2.26686 \atop 0.41390 $ & $ 1.43704 \atop 0.35365 $ & $ 1.34411 \atop 0.35024 $ & $ 1.22740 \atop 0.33833 $ \\[0.5ex] 
${3}$ &  $ 1.08917 \atop 3.97375 $ & $ 1.87215 \atop 4.29139 $ & $ 5.16488 \atop 3.63187 $ & $ 7.11001 \atop 2.47833 $ & $ 13.3072 \atop 1.40486 $ & $ 10.1227 \atop 0.99006 $ & $ 11.0565 \atop 0.80562 $ & $ 7.43739 \atop 0.74920 $ & $ 4.25858 \atop 0.62553 $ & $ 4.56458 \atop 0.61910 $ & $ 3.74400 \atop 0.55211 $ 
\end{tabular}}
\caption{Mean square error of $\hat{\kappa}_{M}$ (top raw) and $\hat{\kappa}_{L}$ (bottom raw) for $\mathrm{GvMF}_{1,3}$ distribution, for different values of $\alpha$ (rows) and $\kappa$ (columns)}
\label{MLET2n1000csm}
\end{table}
\end{center}


\begin{center}
\begin{table}[h]
\footnotesize
{
\begin{tabular}{cccccccccccc}
 & ${0.1}$ & ${0.5}$ & ${1}$  & ${1.5}$ & ${2}$ & ${2.5}$& ${3}$ & ${4}$ & ${5}$ & ${6}$ & ${7}$ \\[0.5ex] 
${0.5}$ & $2.15430 \atop 4.05118 $ & $ 0.04439 \atop 0.03719 $ & $ 0.00892 \atop 0.00571 $ & $ 0.00384 \atop 0.00244 $ & $ 0.00240 \atop 0.00131 $ & $ 0.00200 \atop 0.00111 $ & $ 0.00176 \atop 0.00082 $ & $ 0.00174 \atop 0.00069 $ & $ 0.00208 \atop 0.00073 $ & $ 0.00112 \atop 0.00072 $ & $ 0.00053 \atop 0.00079 $ \\[0.5ex] 
${1}$ & $2.43731 \atop 6.35329 $ & $ 0.11476 \atop 0.15498 $ & $ 0.02662 \atop 0.02641 $ & $ 0.01238 \atop 0.01200 $ & $ 0.00779 \atop 0.00721 $ & $ 0.00620 \atop 0.00555 $ & $ 0.00508 \atop 0.00469 $ & $ 0.00415 \atop 0.00369 $ & $ 0.00424 \atop 0.00357 $ & $ 0.00384 \atop 0.00341 $ & $ 0.00421 \atop 0.00369 $ \\[0.5ex] 
${1.5}$ &  $2.82117 \atop 8.76630 $ & $ 0.20015 \atop 0.30670 $ & $ 0.05429 \atop 0.05952 $ & $ 0.02785 \atop 0.02937 $ & $ 0.01742 \atop 0.01765 $ & $ 0.01432 \atop 0.01473 $ & $ 0.01291 \atop 0.01303 $ & $ 0.00981 \atop 0.01006 $ & $ 0.00956 \atop 0.01004 $ & $ 0.00925 \atop 0.00962 $ & $ 0.01026 \atop 0.01057 $ \\[0.5ex] 
${2}$ &  $2.55262 \atop 9.04940 $ & $ 0.30415 \atop 0.40585 $ & $ 0.09795 \atop 0.09406 $ & $ 0.05169 \atop 0.05020 $ & $ 0.03786 \atop 0.03618 $ & $ 0.02962 \atop 0.02640 $ & $ 0.02579 \atop 0.02466 $ & $ 0.02137 \atop 0.01993 $ & $ 0.02306 \atop 0.02191 $ & $ 0.02314 \atop 0.02198 $ & $ 0.02297 \atop 0.02143 $ \\[0.5ex] 
${2.5}$ &  $2.27668 \atop 8.21909 $ & $ 0.56165 \atop 0.51124 $ & $ 0.12324 \atop 0.12344 $ & $ 0.08668 \atop 0.06832 $ & $ 0.06643 \atop 0.05349 $ & $ 0.05643 \atop 0.04539 $ & $ 0.05414 \atop 0.04304 $ & $ 0.05298 \atop 0.04235 $ & $ 0.04757 \atop 0.03748 $ & $ 0.05055 \atop 0.04087 $ & $ 0.04818 \atop 0.03752 $ \\[0.5ex] 
${3}$ &  $2.64239 \atop 8.45839 $ & $ 0.71927 \atop 0.52430 $ & $ 0.18736 \atop 0.16199 $ & $ 0.14068 \atop 0.09598 $ & $ 0.11660 \atop 0.08051 $ & $ 0.11100 \atop 0.07280 $ & $ 0.11534 \atop 0.07200 $ & $ 0.09450 \atop 0.06493 $ & $ 0.10280 \atop 0.07062 $ & $ 0.10102 \atop 0.06307 $ & $ 0.10246 \atop 0.06369 $ 
\end{tabular}}
\caption{Mean square error of $\hat{\alpha}_{M}$ (top raw) and $\hat{\alpha}_{L}$ (bottom raw) for $\mathrm{GvMF}_{2,3}$ distribution, for different values of $\alpha$ (rows) and $\kappa$ (columns)}
\label{alpha_hat_Type2}
\end{table}
\end{center}




\begin{center}
\begin{table}[h]
\footnotesize
{
\begin{tabular}{cccccccccccc}
 & ${0.1}$ & ${0.5}$ & ${1}$  & ${1.5}$ & ${2}$ & ${2.5}$& ${3}$ & ${4}$ & ${5}$ & ${6}$ & ${7}$ \\[0.5ex] 
${0.5}$ & $0.00308 \atop 0.00264 $ & $ 0.00593 \atop 0.00585 $ & $ 0.00788 \atop 0.00619 $ & $ 0.01080 \atop 0.00815 $ & $ 0.01751 \atop 0.01132 $ & $ 0.03210 \atop 0.02080 $ & $ 0.05088 \atop 0.02932 $ & $ 0.14597 \atop 0.06803 $ & $ 0.39689 \atop 0.16595 $ & $ 0.26720 \atop 0.29599 $ & $ 0.11059 \atop 0.54081 $ \\[0.5ex] 
${1}$ & $0.00290 \atop 0.00299 $ & $ 0.00460 \atop 0.00504 $ & $ 0.00611 \atop 0.00602 $ & $ 0.00912 \atop 0.00885 $ & $ 0.01508 \atop 0.01432 $ & $ 0.02487 \atop 0.02260 $ & $ 0.03871 \atop 0.03695 $ & $ 0.08568 \atop 0.07654 $ & $ 0.18878 \atop 0.16276 $ & $ 0.31815 \atop 0.28533 $ & $ 0.56937 \atop 0.50945 $ \\[0.5ex] 
${1.5}$ &  $0.00317 \atop 0.00352 $ & $ 0.00382 \atop 0.00433 $ & $ 0.00496 \atop 0.00497 $ & $ 0.00712 \atop 0.00715 $ & $ 0.01159 \atop 0.01170 $ & $ 0.02027 \atop 0.02050
 $ & $ 0.03192 \atop 0.03221 $ & $ 0.07540 \atop 0.07713 $ & $ 0.16707 \atop 0.17453 $ & $ 0.28675 \atop 0.29832 $ & $ 0.52890 \atop 0.54837 $ \\[0.5ex] 
${2}$ &  $0.00312 \atop 0.00379 $ & $ 0.00507 \atop 0.00594 $ & $ 0.00512 \atop 0.00515 $ & $ 0.00658 \atop 0.00663 $ & $ 0.01078 \atop 0.01090 $ & $ 0.01733 \atop 0.01729 $ & $ 0.02968 \atop 0.02992 $ & $ 0.07776 \atop 0.07590 $ & $ 0.17039 \atop 0.16294 $ & $ 0.33426 \atop 0.32125 $ & $ 0.54030 \atop 0.52428 $ \\[0.5ex] 
${2.5}$ &  $0.00298 \atop 0.00411 $ & $ 0.00725 \atop 0.00797 $ & $ 0.00705 \atop 0.00761 $ & $ 0.00885 \atop 0.00847 $ & $ 0.01145 \atop 0.01155 $ & $ 0.01878 \atop 0.01850 $ & $ 0.03171 \atop 0.03021 $ & $ 0.08975 \atop 0.07865 $ & $ 0.19040 \atop 0.16314 $ & $ 0.36456 \atop 0.29699 $ & $ 0.65539 \atop 0.52800 $ \\[0.5ex] 
${3}$ &  $0.00277 \atop 0.00449 $ & $ 0.01072 \atop 0.01004 $ & $ 0.01052 \atop 0.01033 $ & $ 0.01256 \atop 0.01179 $ & $ 0.01447 \atop 0.01420 $ & $ 0.02044 \atop 0.02058 $ & $ 0.03274 \atop 0.03070 $ & $ 0.09333 \atop 0.07711 $ & $ 0.24011 \atop 0.17748 $ & $ 0.45512 \atop 0.30782 $ & $ 0.87428 \atop 0.52693 $ 
\end{tabular}}
\caption{Mean square error of $\hat{\kappa}_{M}$ (top raw) and $\hat{\kappa}_{L}$ (bottom raw) for $\mathrm{GvMF}_{2,3}$ distribution, for different values of $\alpha$ (rows) and $\kappa$ (columns)}
\label{kappa_hat_Type2}
\end{table}
\end{center}


\begin{center}
\begin{table}[h]
\footnotesize
{
\begin{tabular}{cccccccccccc}
 & ${0.1}$ & ${0.5}$ & ${1}$  & ${1.5}$ & ${2}$ & ${2.5}$& ${3}$ & ${4}$ & ${5}$ & ${6}$ & ${7}$ \\[0.5ex] 
${0.5}$ & $2.29602 \atop 3.79114 $ & $ 0.18987 \atop 0.16463 $ & $ 0.05705 \atop 0.04162 $ & $ 0.02550 \atop 0.02490 $ & $ 0.01188 \atop 0.01907 $ & $ 0.00814 \atop 0.01631 $ & $ 0.00512 \atop 0.01637 $ & $ 0.00236 \atop 0.02073 $ & $ 0.00128 \atop 0.02427 $ & $ 0.00045 \atop 0.02521 $ & $ 0.00045 \atop 0.02548 $ \\[0.5ex] 
${1}$ & $1.89494 \atop 6.47041 $ & $ 0.72857 \atop 1.99071 $ & $ 0.23713 \atop 0.21501 $ & $ 0.11809 \atop 0.08130 $ & $ 0.07490 \atop 0.04956 $ & $ 0.05300 \atop 0.03309 $ & $ 0.03753 \atop 0.02432 $ & $ 0.02852 \atop 0.02189 $ & $ 0.01925 \atop 0.01923 $ & $ 0.01203 \atop 0.02272 $ & $ 0.00443 \atop 0.02594 $ \\[0.5ex] 
${1.5}$ &  $2.09251 \atop 12.7767 $ & $ 1.41746 \atop 7.55889 $ & $ 0.69305 \atop 1.11221 $ & $ 0.28007 \atop 0.28031 $ & $ 0.16580 \atop 0.14598 $ & $ 0.09577 \atop 0.08200 $ & $ 0.07204 \atop 0.06144 $ & $ 0.04872 \atop 0.04110 $ & $ 0.03733 \atop 0.03296 $ & $ 0.03224 \atop 0.02763 $ & $ 0.02760 \atop 0.02579 $ \\[0.5ex] 
${2}$ &  $2.18619 \atop 18.4855 $ & $ 2.16193 \atop 11.8352 $ & $ 1.34087 \atop 4.29336 $ & $ 0.66370 \atop 1.22665 $ & $ 0.37039 \atop 0.40504 $ & $ 0.20955 \atop 0.20974 $ & $ 0.15980 \atop 0.16143 $ & $ 0.09065 \atop 0.08807 $ & $ 0.05605 \atop 0.05484 $ & $ 0.04741 \atop 0.04622 $ & $ 0.03871 \atop 0.03623 $ \\[0.5ex] 
${2.5}$ &  $3.15268 \atop 20.7207 $ & $ 3.14653 \atop 17.0535 $ & $ 2.42570 \atop 7.75701 $ & $ 1.51950 \atop 3.11424 $ & $ 0.79157 \atop 1.16675 $ & $ 0.47397 \atop 0.52106 $ & $ 0.33253 \atop 0.34421 $ & $ 0.17312 \atop 0.18537 $ & $ 0.11837 \atop 0.12071 $ & $ 0.08002 \atop 0.08197 $ & $ 0.06680 \atop 0.06698 $ \\[0.5ex] 
${3}$ &  $4.56639 \atop 20.2178 $ & $ 4.87133 \atop 17.9717 $ & $ 4.04307 \atop 10.7576 $ & $ 2.64681 \atop 5.75811 $ & $ 1.72811 \atop 2.78415 $ & $ 1.07167 \atop 1.28941 $ & $ 0.70276 \atop 0.74024 $ & $ 0.33226 \atop 0.35797 $ & $ 0.20734 \atop 0.22610 $ & $ 0.15167 \atop 0.15511 $ & $ 0.11787 \atop 0.12132 $ 
\end{tabular}}
\caption{Mean square error of $\hat{\alpha}_{M}$ (top raw) and $\hat{\alpha}_{L}$ (bottom raw) for $\mathrm{GvMF}_{3,3}$ distribution, for different values of $\alpha$ (rows) and $\kappa$ (columns)}
\label{alpha_type3}
\end{table}
\end{center}


\begin{center}
\begin{table}[h]
\footnotesize
{
\begin{tabular}{cccccccccccc}
 & ${0.1}$ & ${0.5}$ & ${1}$  & ${1.5}$ & ${2}$ & ${2.5}$& ${3}$ & ${4}$ & ${5}$ & ${6}$ & ${7}$ \\[0.5ex] 
${0.5}$ & $0.29970 \atop 0.60213 $ & $ 0.06602 \atop 0.06010 $ & $ 0.05435 \atop 0.04156 $ & $ 0.05186 \atop 0.04885 $ & $ 0.04292 \atop 0.05101 $ & $ 0.04582 \atop 0.05939 $ & $ 0.04249 \atop 0.06571 $ & $ 0.05097 \atop 0.10154 $ & $ 0.05318 \atop 0.12560 $ & $ 0.06447 \atop 0.15621 $ & $ 0.08802 \atop 0.18777 $ \\[0.5ex] 
${1}$ & $0.19968 \atop 0.89717 $ & $ 0.14994 \atop 0.64104 $ & $ 0.12997 \atop 0.14258 $ & $ 0.13578 \atop 0.10734 $ & $ 0.14216 \atop 0.10834 $ & $ 0.14670 \atop 0.10368 $ & $ 0.12929 \atop 0.09266 $ & $ 0.14336 \atop 0.11498 $ & $ 0.14140 \atop 0.12950 $ & $ 0.14276 \atop 0.17055 $ & $ 0.10894 \atop 0.19490 $ \\[0.5ex] 
${1.5}$ &  $0.18726 \atop 2.12072 $ & $ 0.22813 \atop 1.96361 $ & $ 0.22610 \atop 0.54431 $ & $ 0.21028 \atop 0.25810 $ & $ 0.22102 \atop 0.22148 $ & $ 0.19741 \atop 0.18410 $ & $ 0.19829 \atop 0.17962 $ & $ 0.21212 \atop 0.18814 $ & $ 0.23031 \atop 0.21156 $ & $ 0.24706 \atop 0.21623 $ & $ 0.24683 \atop 0.23138 $ \\[0.5ex] 
${2}$ &  $0.13815 \atop 3.07854 $ & $ 0.15945 \atop 	2.32459 $ & $ 0.29881 \atop 1.70401 $ & $ 0.35199 \atop 0.86157 $ & $ 0.36035 \atop 0.45843 $ & $ 0.31712 \atop 0.35382 $ & $ 0.33244 \atop 0.35891 $ & $ 0.31035 \atop 0.30941 $ & $ 0.30316 \atop 0.30274 $ & $ 0.32061 \atop 0.31870 $ & $ 0.31060 \atop 0.29598 $ \\[0.5ex] 
${2.5}$ &  $0.17697 \atop 3.89041 $ & $ 0.15577 \atop 3.15115 $ & $ 0.37208 \atop 2.58508 $ & $ 0.55335 \atop 1.76622 $ & $ 0.55512 \atop 0.99471 $ & $ 0.51438 \atop 0.65380 $ & $ 0.51368 \atop 0.58423 $ & $ 0.48383 \atop 0.53832 $ & $ 0.49249 \atop 0.50869 $ & $ 0.42680 \atop 0.43904 $ & $ 0.47090 \atop 0.47885 $ \\[0.5ex] 
${3}$ &  $0.17079 \atop 4.23216 $ & $ 0.14980 \atop 3.34039 $ & $ 0.44026 \atop 3.09945 $ & $ 0.67499 \atop 2.79798 $ & $ 0.86090 \atop 1.99287 $ & $ 0.84880 \atop 1.32447 $ & $ 0.85450 \atop 1.07203 $ & $ 0.69410 \atop 0.79730 $ & $ 0.67099 \atop 0.75204 $ & $ 0.66948 \atop 0.69418 $ & $ 0.68146 \atop 0.69721 $
\end{tabular}}
\caption{Mean square error of $\hat{\kappa}_{M}$ (top raw) and $\hat{\kappa}_{L}$ (bottom raw) for $\mathrm{GvMF}_{3,3}$ distribution, for different values of $\alpha$ (rows) and $\kappa$ (columns)}
\label{kappa_type3}
\end{table}
\end{center}



\begin{figure}[h]
     \centering
     \begin{subfigure}[b]{0.32\textwidth}
         \centering
         \includegraphics[width=\textwidth]{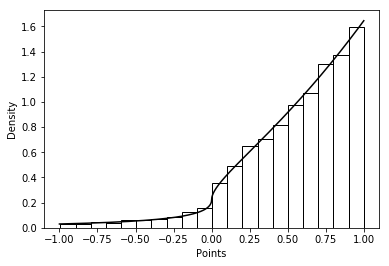}
         \caption{$\kappa = 1$}
         \label{fig:a_05_k_1}
     \end{subfigure}
     \hfill
     \begin{subfigure}[b]{0.32\textwidth}
         \centering
         \includegraphics[width=\textwidth]{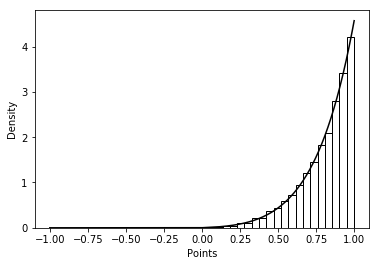}
         \caption{$\kappa = 4$}
         \label{fig:a_05_k_4}
     \end{subfigure}
     \hfill
     \begin{subfigure}[b]{0.32\textwidth}
         \centering
         \includegraphics[width=\textwidth]{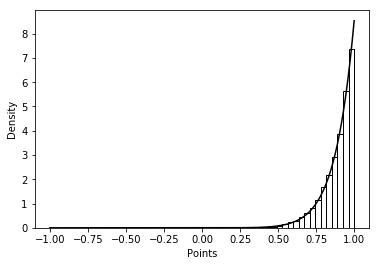}
         \caption{$\kappa = 8$}
         \label{fig:a_05_k_8}
         \end{subfigure}
    \caption{Density $f_1$ from \eqref{f1} for $\alpha =0.5$ and  different values of $\kappa $.}
    \label{fig:type1_a_05}
        
\end{figure}

\begin{figure}[h]
     \centering
     \begin{subfigure}[b]{0.32\textwidth}
         \centering
         \includegraphics[width=\textwidth]{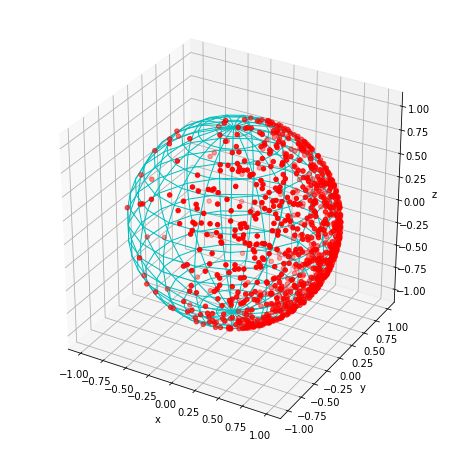}
         \caption{$\kappa = 1$}
         \label{fig:Type1a_05_k_1_sphere}
     \end{subfigure}
     \hfill
     \begin{subfigure}[b]{0.32\textwidth}
         \centering
         \includegraphics[width=\textwidth]{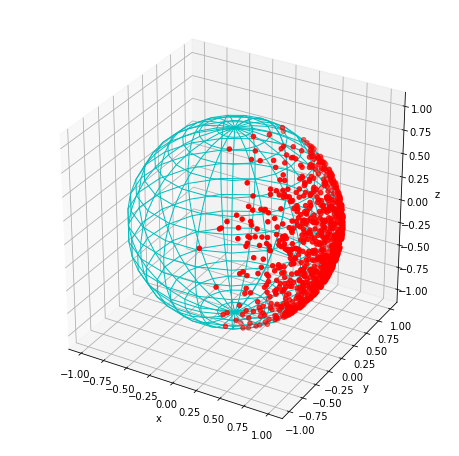}
         \caption{$\kappa = 4$}
         \label{fig:Type1a_05_k_4_sphere}
     \end{subfigure}
     \hfill
     \begin{subfigure}[b]{0.32\textwidth}
         \centering
         \includegraphics[width=\textwidth]{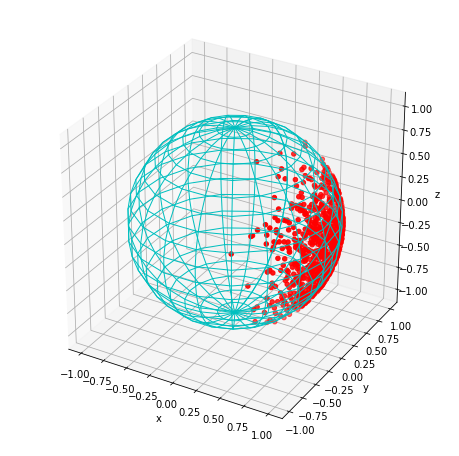}
         \caption{$\kappa = 8$}
         \label{fig:Type1a_05_k_8_sphere}
     \end{subfigure}
    \caption{Realisations of $X\sim \mathrm{GvMF_{1,3}}(\kappa,\alpha,\mu)$ with  $\alpha =0.5$ and  different values of $\kappa$.}
    \label{fig:S:type1_a_05}
        
\end{figure}




\begin{figure}[h]
     \centering
     \begin{subfigure}[b]{0.32\textwidth}
         \centering
         \includegraphics[width=\textwidth]{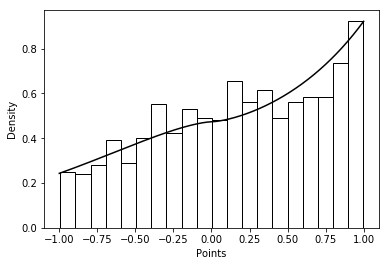}
         \caption{$\kappa = 1$}
         \label{fig:a_15_k_1}
     \end{subfigure}
     \begin{subfigure}[b]{0.32\textwidth}
         \centering
         \includegraphics[width=\textwidth]{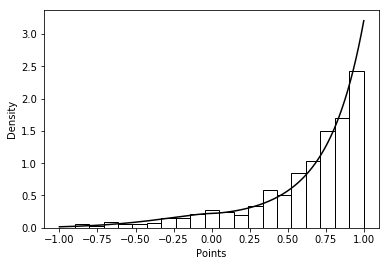}
         \caption{$\kappa = 4$}
         \label{fig:a_15_k_4}
     \end{subfigure}
     \hfill
     \begin{subfigure}[b]{0.32\textwidth}
         \centering
         \includegraphics[width=\textwidth]{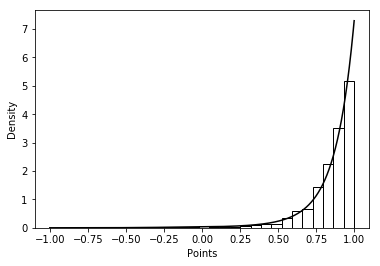}
         \caption{$\kappa = 8$}
         \label{fig:a_15_k_8}
         \end{subfigure}
    \caption{Density $f_1$ from \eqref{f1} for $\alpha =1.5$ and  different values of $\kappa $.}
    \label{fig:type1_a_15}
\end{figure}


\begin{figure}[h]
     \centering
     \begin{subfigure}[b]{0.32\textwidth}
         \centering
         \includegraphics[width=\textwidth]{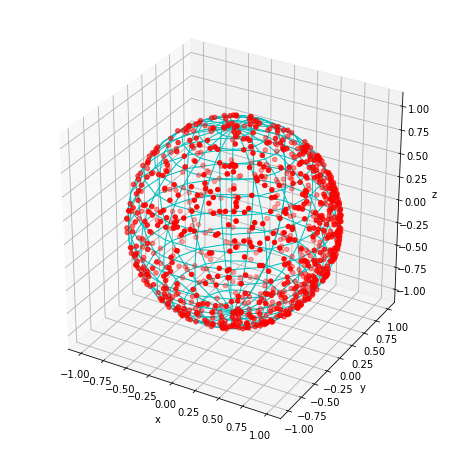}
         \caption{$\kappa = 1$}
         \label{fig:Type1a_15_k_1_sphere}
     \end{subfigure}
     \hfill
     \begin{subfigure}[b]{0.32\textwidth}
         \centering
         \includegraphics[width=\textwidth]{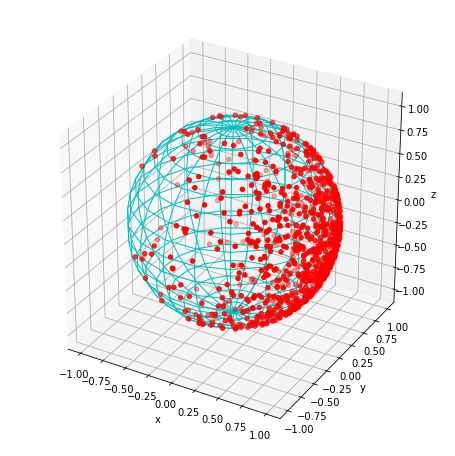}
         \caption{$\kappa = 4$}
         \label{fig:Type1a_15_k_4_sphere}
     \end{subfigure}
     \hfill
     \begin{subfigure}[b]{0.32\textwidth}
         \centering
         \includegraphics[width=\textwidth]{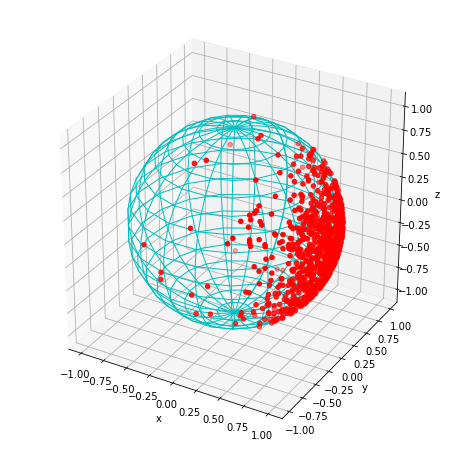}
         \caption{$\kappa = 8$}
         \label{fig:Type1a_15_k_8}
     \end{subfigure}
    \caption{Realisations of $X\sim \mathrm{GvMF_{1,3}}(\kappa,\alpha,\mu)$ with  $\alpha =1.5$ and  different values of $\kappa$.}
    \label{fig:type1_a_15_sphere}
        
\end{figure}



\begin{figure}[h]
     \centering
     \begin{subfigure}[b]{0.32\textwidth}
         \centering
         \includegraphics[width=\textwidth]{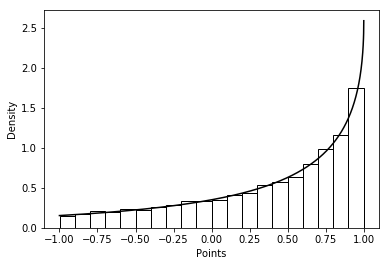}
         \caption{$\kappa = 1$}
         \label{fig:Type2a_05_k_1}
     \end{subfigure}
     \hfill
     \begin{subfigure}[b]{0.32\textwidth}
         \centering
         \includegraphics[width=\textwidth]{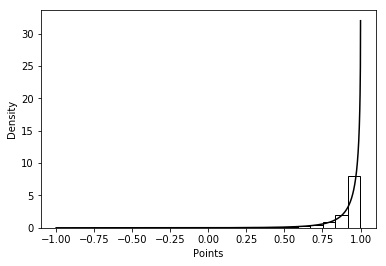}
         \caption{$\kappa = 4$}
         \label{fig:Type2a_05_k_4}
     \end{subfigure}
     \hfill
     \begin{subfigure}[b]{0.32\textwidth}
         \centering
         \includegraphics[width=\textwidth]{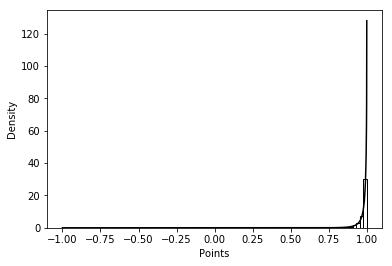}
         \caption{$\kappa = 8$}
         \label{fig:Type2a_05_k_8}
         \end{subfigure}
    \caption{Density $f_2$ from \eqref{f2} for $\alpha = 0.5$ and  different values of $\kappa $.}
    \label{fig:type2_a_05}
\end{figure}


\begin{figure}[h]
     \centering
     \begin{subfigure}[b]{0.32\textwidth}
         \centering
         \includegraphics[width=\textwidth]{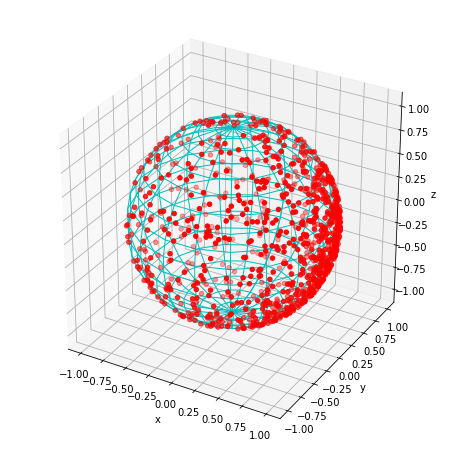}
         \caption{$\kappa = 1$}
         \label{fig:Type2a_05_k_1_sphere}
     \end{subfigure}
     \hfill
     \begin{subfigure}[b]{0.32\textwidth}
         \centering
         \includegraphics[width=\textwidth]{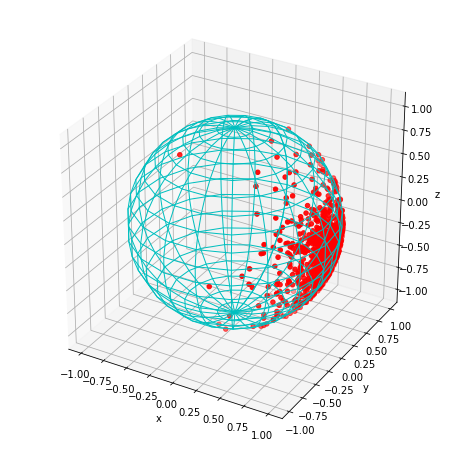}
         \caption{$\kappa = 4$}
         \label{fig:Type2a_05_k_4_sphere}
     \end{subfigure}
     \hfill
     \begin{subfigure}[b]{0.32\textwidth}
         \centering
         \includegraphics[width=\textwidth]{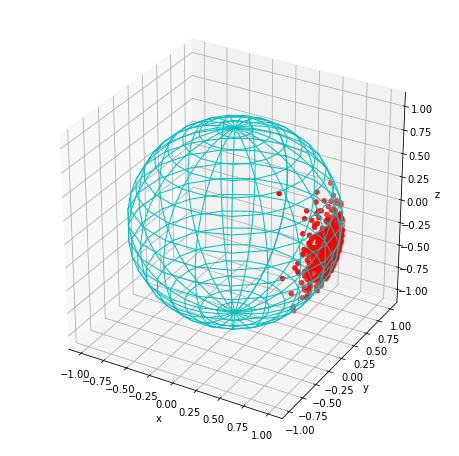}
         \caption{$\kappa = 8$}
         \label{fig:Type2a_05_k_8_sphere}
     \end{subfigure}
    \caption{Realisations of $X\sim \mathrm{GvMF_{2,3}}(\kappa,\alpha,\mu)$ with  $\alpha =0.5$ and  different values of $\kappa$.}
    \label{fig:type2_a_05_sphere}
        
\end{figure}


\begin{figure}[h]
     \centering
     \begin{subfigure}[b]{0.32\textwidth}
         \centering
         \includegraphics[width=\textwidth]{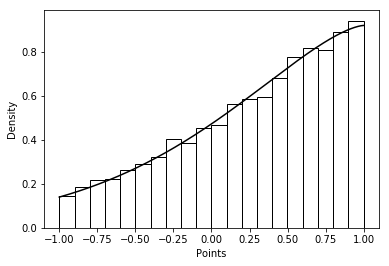}
         \caption{$\kappa = 1$}
         \label{fig:Type2a_15_k_1}
     \end{subfigure}
     \hfill
     \begin{subfigure}[b]{0.32\textwidth}
         \centering
         \includegraphics[width=\textwidth]{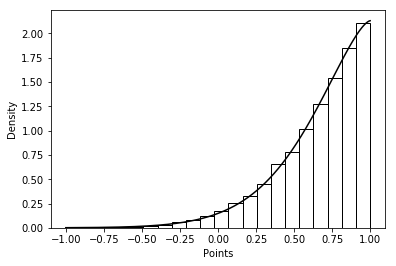}
         \caption{$\kappa = 4$}
         \label{fig:Type2a_15_k_4}
     \end{subfigure}
     \hfill
     \begin{subfigure}[b]{0.32\textwidth}
         \centering
         \includegraphics[width=\textwidth]{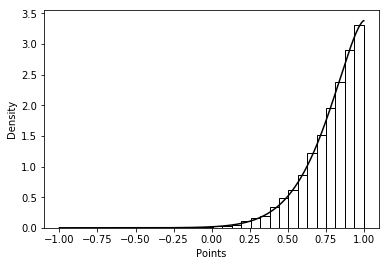}
         \caption{$\kappa = 8$}
         \label{fig:Type2a_15_k_8}
         \end{subfigure}
    \caption{Density $f_2$ from \eqref{f2} for $\alpha = 1.5$ and  different values of $\kappa $.}
    \label{fig:type2_a_15}
        
\end{figure}


\begin{figure}[h]
     \centering
     \begin{subfigure}[b]{0.32\textwidth}
         \centering
         \includegraphics[width=\textwidth]{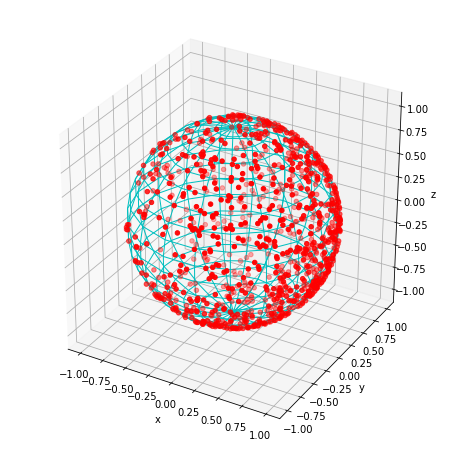}
         \caption{$\kappa = 1$}
         \label{fig:Type2a_15_k_1_sphere}
     \end{subfigure}
     \hfill
     \begin{subfigure}[b]{0.32\textwidth}
         \centering
         \includegraphics[width=\textwidth]{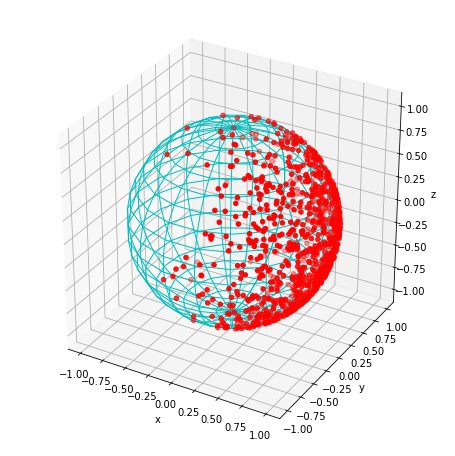}
         \caption{$\kappa = 4$}
         \label{fig:Type2a_15_k_4_sphere}
     \end{subfigure}
     \hfill
     \begin{subfigure}[b]{0.32\textwidth}
         \centering
         \includegraphics[width=\textwidth]{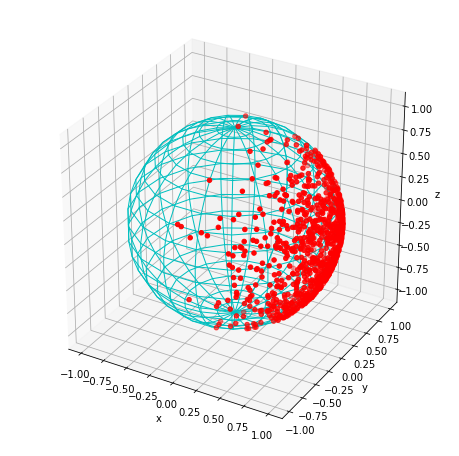}
         \caption{$\kappa = 8$}
         \label{fig:Type2a_15_k_8_sphere}
     \end{subfigure}
    \caption{Realisations of $X\sim \mathrm{GvMF_{2,3}}(\kappa,\alpha,\mu)$ with  $\alpha =1.5$ and  different values of $\kappa$.}
    \label{fig:type2_a_15_sphere}
        
\end{figure}



\begin{figure}[h]
     \centering
     \begin{subfigure}[b]{0.32\textwidth}
         \centering
         \includegraphics[width=\textwidth]{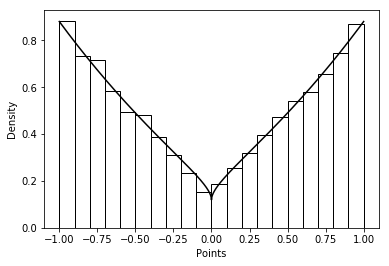}
         \caption{$\kappa = 1$}
         \label{fig:Type3a_05_k_1}
     \end{subfigure}
     \hfill
     \begin{subfigure}[b]{0.32\textwidth}
         \centering
         \includegraphics[width=\textwidth]{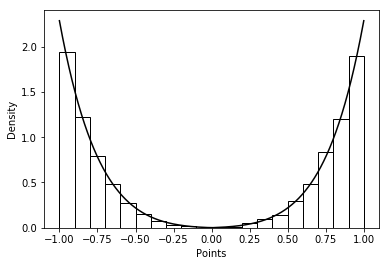}
         \caption{$\kappa = 4$}
         \label{fig:Type3a_05_k_4}
     \end{subfigure}
     \hfill
     \begin{subfigure}[b]{0.32\textwidth}
         \centering
         \includegraphics[width=\textwidth]{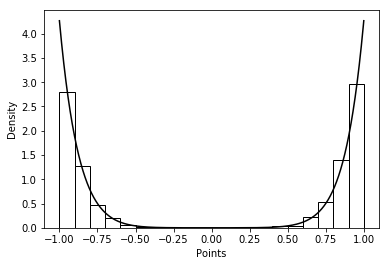}
         \caption{$\kappa = 8$}
         \label{fig:Type3a_05_k_8}
         \end{subfigure}
    \caption{Density $f_3$ from \eqref{f3} for $\alpha = 0.5$ and  different values of $\kappa$.}
    \label{fig:type3_a_05}
        
\end{figure}


\begin{figure}[h]
     \centering
     \begin{subfigure}[b]{0.32\textwidth}
         \centering
         \includegraphics[width=\textwidth]{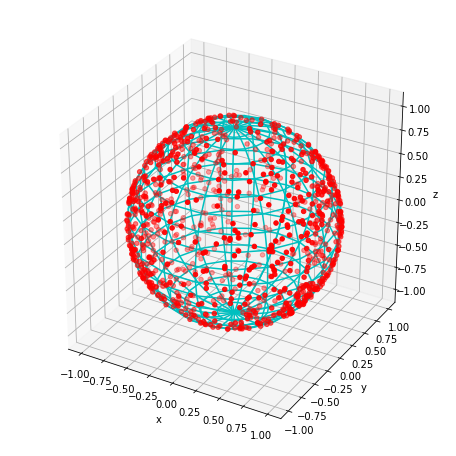}
         \caption{$\kappa = 1$}
         \label{fig:Type3a_05_k_1_sphere}
     \end{subfigure}
     \hfill
     \begin{subfigure}[b]{0.32\textwidth}
         \centering
         \includegraphics[width=\textwidth]{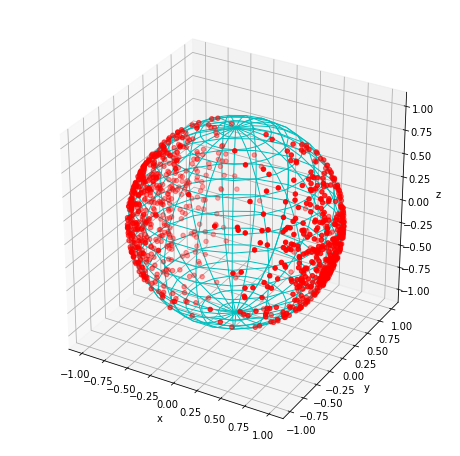}
         \caption{$\kappa = 4$}
         \label{fig:Type3a_05_k_4_sphere}
     \end{subfigure}
     \hfill
     \begin{subfigure}[b]{0.32\textwidth}
         \centering
         \includegraphics[width=\textwidth]{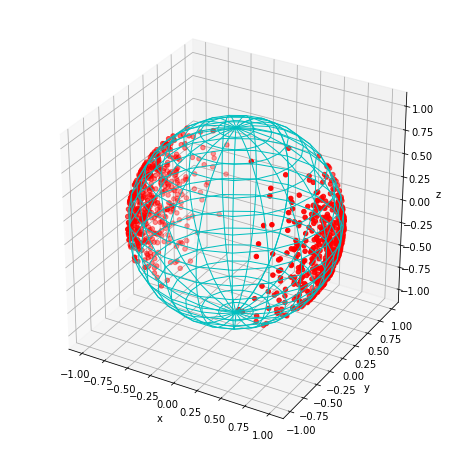}
         \caption{$\kappa = 8$}
         \label{fig:Type3a_05_k_8_sphere}
     \end{subfigure}
    \caption{Realisations of $X\sim \mathrm{GvMF_{3,3}}(\kappa,\alpha,\mu)$ with  $\alpha =0.5$ and  different values of $\kappa$.}
    \label{fig:type3_a_05_sphere}
        
\end{figure}


\begin{figure}[h]
     \centering
     \begin{subfigure}[b]{0.32\textwidth}
         \centering
         \includegraphics[width=\textwidth]{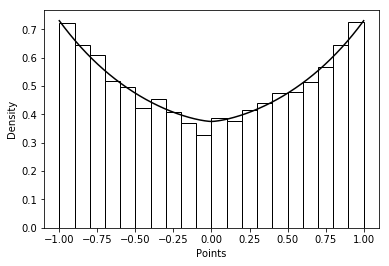}
         \caption{$\kappa = 1$}
         \label{fig:Type3a_15_k_1}
     \end{subfigure}
     \hfill
     \begin{subfigure}[b]{0.32\textwidth}
         \centering
         \includegraphics[width=\textwidth]{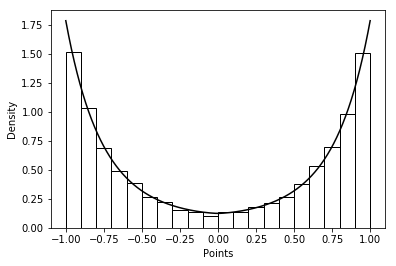}
         \caption{$\kappa = 4$}
         \label{fig:Type3a_15_k_4}
     \end{subfigure}
     \hfill
     \begin{subfigure}[b]{0.32\textwidth}
         \centering
         \includegraphics[width=\textwidth]{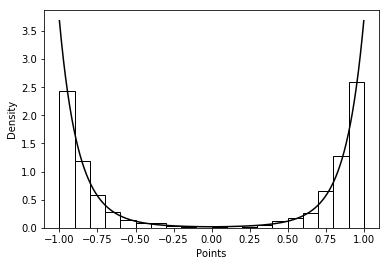}
         \caption{$\kappa = 8$}
         \label{fig:Type3a_15_k_8}
         \end{subfigure}
    \caption{Density $f_3$ from \eqref{f3} for $\alpha = 1.5$ and  different values of $\kappa$.}
    \label{fig:type3_a_15}
        
\end{figure}


\begin{figure}[h]
     \centering
     \begin{subfigure}[b]{0.32\textwidth}
         \centering
         \includegraphics[width=\textwidth]{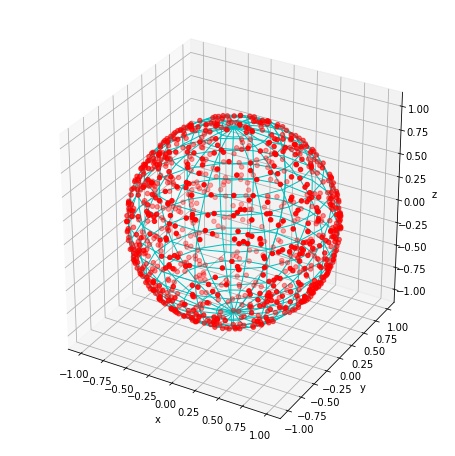}
         \caption{$\kappa = 1$}
         \label{fig:Type3a_15_k_1_sphere}
     \end{subfigure}
     \hfill
     \begin{subfigure}[b]{0.32\textwidth}
         \centering
         \includegraphics[width=\textwidth]{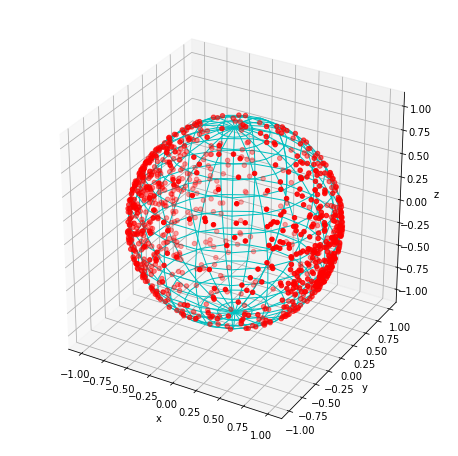}
         \caption{$\kappa = 4$}
         \label{fig:Type3a_15_k_2_sphere}
     \end{subfigure}
     \hfill
     \begin{subfigure}[b]{0.32\textwidth}
         \centering
         \includegraphics[width=\textwidth]{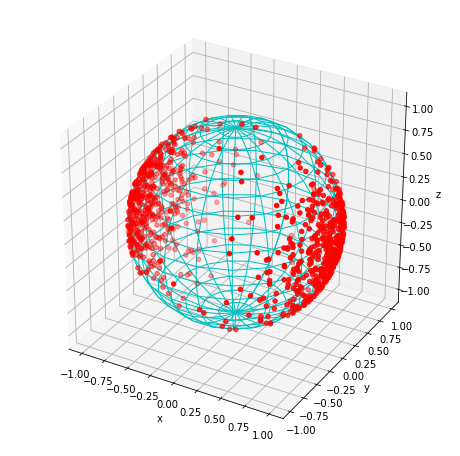}
         \caption{$\kappa = 8$}
         \label{fig:Type3a_15_k_4_sphere}
     \end{subfigure}
    \caption{Realisations of $X\sim \mathrm{GvMF_{3,3}}(\kappa,\alpha,\mu)$ with  $\alpha =1.5$ and  different values of $\kappa$.}
    \label{fig:type3_a_15_sphere}
        
\end{figure}

\end{document}